\documentclass{amsart}

\newtheorem{theorem}{Theorem}
\newtheorem{proposition}[theorem]{Proposition}
\newtheorem{lemma}[theorem]{Lemma}
\newtheorem{corollary}[theorem]{Corollary}

\newtheorem{definition}[theorem]{Definition}

\theoremstyle{remark}
\newtheorem{remark}[theorem]{\bf {Remark}}
\newtheorem{example}[theorem]{\bf {Example}}

\numberwithin{equation}{section}
\usepackage{amsmath}
\usepackage{amssymb}
\usepackage{color}
\usepackage[top=2.8cm, bottom=2.5cm, left=2.2cm, right=2.4cm] {geometry}

\newcommand{\Nset}{\mathbb{N}}
\newcommand{\Qset}{\mathbb{Q}}
\newcommand{\Rset}{\mathbb{R}}

\title{The Samuel Realcompactification of a metric space}
\author{M. Isabel Garrido}
\address{Instituto de Matem\'{a}tica Interdisciplinar (IMI), Departamento de Geometr\'{\i}a y  Topolog\'{\i}a, Universidad Complutense de Madrid, 28040 Madrid, Spain}
\email{maigarri@math.ucm.es}
\thanks{Partially supported by MINECO Project MTM2012-34341 (Spain).}

\author{Ana S. Mero\~{n}o}
\address{Departamento de An\'{a}lisis Matem\'{a}tico, Universidad Complutense de Madrid, 28040-Madrid, Spain.}
\email{anasoledadmerono@ucm.es}


\begin{document}


\begin{abstract}
In this paper we introduce a realcompactification  for  any metric space $(X,d)$, defined by means of the family of all its real-valued  uniformly continuous functions. We call it the Samuel realcompactification, according to the well known Samuel compactification  associated to the family of all the  bounded real-valued uniformly continuous functions.  Among many other things, we study the corresponding problem of the Samuel realcompactness for metric spaces.  At this respect, we prove that a result of Kat\v{e}tov-Shirota type occurs in this context, where the completeness property is replaced by  Bourbaki-completeness (a notion recently introduced by the authors) and the closed discrete  subspaces are replaced  by the  uniformly discrete ones. More precisely,  we  see   that a metric space $(X,d)$ is Samuel realcompact iff it is Bourbaki-complete and every uniformly discrete subspace of $X$ has non-measurable cardinal. As a consequence, we derive that  a normed space is Samuel realcompact iff it has finite dimension. And this means in particular that realcompactness  and Samuel realcompactness can be very  far apart. The paper also contains results relating this  realcompactification with the so-called Lipschitz realcompactification (also studied here), with the classical Hewitt-Nachbin realcompactification and with the completion of the initial metric space.

\end{abstract}

\maketitle

\large

\section*{Introduction}

\smallskip

In this paper we are going to introduce a realcompactification that can be defined for any metric space as well as for any uniform space. It represents a way of extending the classical topological notion of realcompactness to the frame of uniform spaces. Moreover, it  will be  related to the well known Samuel compactification, in the same way that the Hewitt-Nachbin realcompactification of a completely regular space is related to its Stone-\v{C}ech compactification.  Some authors, like Ginsburg, Isbell, Rice, Reynolds, Hu\v{s}ek and Pulgar\'{\i}n,   considered some equivalent forms of this  uniform realcompactification when they studied the completion of certain uniformities defined on a space (see \cite{ginsburg}, \cite{isbell}, \cite{rice}, \cite{reynolds} and \cite{husek}). Nevertheless, we will be interested here in the use of just the set of the  real-valued  uniformly continuous functions, instead of considering several families of uniform covers which define different uniformities on the space. In this line, Nj{\aa}stad  studied in \cite{najstad} an extension of the notion of realcompactness for proximity spaces. We will see later that this definition   coincides with our notion, at least for metric spaces. Another special uniform extension for uniform spaces was studied by Curzer and Hager in \cite{curzer}, and more recently by Chekeev in \cite{chekeev}. We will see  that this last extension is nothing but  usual realcompactness, for metric spaces.

\smallskip

Recall that in 1948, P. Samuel (\cite{samuel}) defined the compactification, which bears his name, for any uniform space $X$  by means of some kind of ultrafilters in $X$.  We  refer to the nice article  by Woods (\cite{woods}) where several characterizations and properties of this compactification are given in the special case of metric spaces. Even if many concepts and results contained here admit easy generalizations to uniform spaces, we will work in the realm of metric spaces, mainly because in this way we will have the useful tools of some  classes of  Lipschitz functions.

\smallskip

So, let  $(X,d)$ be  a metric space and let us denote by $s_d X$ its Samuel compactification. It is known that $s_d X$ can be characterized as the smallest compactification (considering the usual order in the family of all compactifications of $X$) with the property that each bounded real-valued uniformly continuous  function  on $X$ can be continuously extended to $s_d X$ (see for instance \cite{woods}). Inspired by this fact, we are going to see that the smallest realcompactification with the property that every (not necessarily bounded) real-valued uniformly continuous function on $X$ admits a continuous extension to it, also exists. We will call it the {\it Samuel realcompactification} of $X$ and it will be denoted by $H(U_d(X))$
(the letter $H$ come from the word ``homomorphism'' as we will see in the next section).

\smallskip

On the other hand, it is known that  $s_d X$  is also the smallest compactification to which each bounded real-valued Lipschitz function on $X$ can be continuously extended (see for instance \cite{garrido1}). Therefore, we can say that both families of bounded real-valued functions give  the same (equivalent) compactification of $X$. A natural question would be whether (not necessarily bounded)  Lipschitz functions also determine the Samuel realcompactification. We will see that this is not the case,  having then another realcompactification  that we will call the  {\it Lipschitz realcompactification} of $X$, denoted by $H(Lip_d(X))$.

\smallskip

This paper is mainly devoted to study both realcompactifications for a metric space $(X,d)$, and the contents are as follows.
First we analyze  the Lispchitz realcompactification, and we characterize those metric spaces such that  $X= H(Lip_d(X))$, which we call  {\it Lipschitz realcompact}. Then, we show that Lipschitz realcompactifications are a key part of the Samuel realcompactifications since we  prove that $H(U_d(X))$ is  the supremum of all the Lipschitz realcompactifications given by all the uniformly equivalent metrics, i.e,  $$H(U_d (X))=\bigvee \big\{H(Lip_\rho(X)): \rho\stackrel{{u}}{\sim} d\big\}.$$

\smallskip

Next, we will address the problem of the Samuel realcompactness for a metric space. We say that a metric space $(X,d)$ is  {\it Samuel realcompact} whenever    $X=H(U_d(X))$. Our main result in this line will be a theorem of Kat\v{e}tov-Shirota type, since it involves some kind of completeness and some hypothesis about non-measurable cardinals. More precisely, we will prove that a metric space is Samuel realcompact iff it is Bourbaki-complete and every uniformly discrete subspace has non-measurable cardinal. Recall that the property of {\it Bourbaki-completeness} relies between compactness and completeness,  and it was recently  introduced and studied  by us in \cite{merono2}.

\smallskip

Reached this point of the paper, we will observe that most of the results  obtained up to here are very related to some families of bounded subsets of the metric space, and more precisely related  to some bornologies on $X$. Recall that a family of subsets of $X$ is said to be a {\it bornology} whenever they form a cover of $X$, closed by finite unions, and stable by subsets. Thus, we will present in a synoptic table the coincidence of some extensions of the metric space with the equality between some special metric bornologies.  Concerning to  realcompactifications given by bornologies, we refer to the paper by Vroegrijk  \cite{vroegrijk}, where the so-called {\it bornological realcompactifications} for general topological spaces are studied.

\smallskip

Finally, we will devote last section to compare the Samuel realcompactification of a metric space $(X,d)$ with  the classical Hewitt-Nachbin realcompactification $\upsilon X$. For that reason, results contained here will have a more topological flavor. We will check that $\upsilon X$ not only lives in $\beta X$ (the Stone-\v{C}ech compactification) but also in $s_d X$. And, we will see that, for any  topological metrizable space $(X,\tau)$, $$\upsilon X =\bigvee \big\{H(U_d(X)): d\text{ metric with \,}  \tau_d=\tau\}=\bigvee \big\{H(Lip_d(X)): d\text{ metric with \,}  \tau_d=\tau\}.$$

\smallskip

\section{Preliminaries on realcompactifications}

\smallskip

Most of the results contained in this section can be seen, for instance, in \cite{garrido1}.
For a Tychonoff topological space $X$ and for a family $\mathcal L$ of real-valued continuous functions, that we suppose having the algebraic structure of unital vector lattice, we denote by $H(\mathcal L)$ the set of all the real unital vector lattice homomorphisms on $\mathcal L$. We consider on $H(\mathcal L)$ the topology inherited as a subspace of the product space ${\mathbb R}^{\mathcal L}$, where the real line $\mathbb R$ is endowed with the usual topology. It is easy to check that $H(\mathcal L)$ is closed in ${\mathbb R}^{\mathcal L}$, and then it is a realcompact space. In the same way, we can consider $\mathcal L^*$ the unital vector sublattice formed by the bounded functions in $\mathcal L$. Now the space $H(\mathcal L^*)$ is in fact compact, and it is easy to see that  $H(\mathcal L)$ can be considered as a topological subspace of $H(\mathcal L^*)$. Hence, we can write $H(\mathcal L)\subset H(\mathcal L^*).$

\smallskip

Moreover, when the family $\mathcal L$ separates points a closed sets of $X$, i.e., when for every  closed subset $F$ of $X$ and $x\in X\setminus F$ there exists some $f\in \mathcal L$ such that $f(x)\notin \overline{f(F)}$, then we can embed the topological space $X$ (in a densely way) in  $H(\mathcal L)$ and also in $H(\mathcal L^*)$. And this means, in particular, that $H(\mathcal L)$  is a realcompactification of $X$ and $H(\mathcal L^*)$ is a compactification of $X$. And then, we have $$X\subset H(\mathcal L)\subset H(\mathcal L^*).$$

On the other hand,   every function in $\mathcal L$ (respectively in $\mathcal L^*$) admits a unique  continuous  extension to $H(\mathcal L)$ (resp. to $H(\mathcal L^*)$). In fact, $H(\mathcal L)$ (resp. $H(\mathcal L^*)$) is characterized  (up to equivalence) as the smallest realcompactification (resp. compactification) of $X$ with this property. Note that we are here considering the usual order in the set of all the realcompactifications and compactifications on $X$. Namely, we say that $\alpha_1 X\leq \alpha_2 X$ whenever there is a continuous mapping $h:\alpha_2X\to \alpha_1X$ leaving $X$ pointwise fixed. And we say that $\alpha_1 X$ and $\alpha_2X$  are equivalent whenever $\alpha_1 X\leq \alpha_2 X$ and $\alpha_2 X\leq \alpha_1 X$, and this implies the existence of a homeomorphism between $\alpha_1 X$ and $\alpha_2 X$ leaving $X$ pointwise fixed.

\smallskip

A very useful property in connection with the extension of continuous functions is the following: ``each $f\in \mathcal L$ can be extended to a unique continuous function $f^*: H(\mathcal L^*)\to {\mathbb R} \cup \{\infty\}$, where ${\mathbb R} \cup \{\infty\}$ denotes the one point compactification of the real line'' (see \cite{garrido1}). In particular, this allows us to describe the space $H(\mathcal L)$ as follows, $$H(\mathcal L)=\big\{\xi\in H(\mathcal L^*): \, f^*(\xi)\neq \infty \text{\,\, for all \,} f\in \mathcal L\big\}.$$

Note that if we consider $\mathcal L =C(X)$, the set of all the real-valued continuous functions on  $X$, then  $H(C(X))=\upsilon X$ is the Hewit-Nachbin realcompactification of $X$ and $H(C^*(X))=\beta X$ is now its  Stone-\v{C}ech compactification.

\smallskip

When $(X,d)$ is a metric space, two important unital vector lattices of real-valued functions can be also considered. Namely,  the set $Lip_d(X)$ of all the real-valued  Lipschitz functions, and   the set $U_d(X)$ of all the real-valued uniformly continuous functions defined on $X$. At this point we can say that $H(U_d^*(X))$ is in fact  the   {\it Samuel compactification} $s_d X$ of $X$ since, as we said in the Introduction, this compactification is characterized as the smallest compactification where all the real and bounded uniformly continuous functions on $X$ can be continuously extended (\cite{woods}). Now, according to the fact that $\mathcal L$ and its uniform closure $\mathcal{\overline L}$ define equivalent realcompactifications (see \cite{garrido1}), together with the well known result  from which every bounded and uniformly continuous functions can be uniformly approximated by Lipschitz functions (see for instance \cite{garrido2}), we can derive that $$H(Lip_d^*(X))=  H(U_d^*(X)) = s_d X.$$

Then, we wonder what happen   when we consider unbounded functions. First of all, note that a similar uniform approximation  result does not exist in the unbounded case.  In fact, we know that for a metric space $(X,d)$ the family $Lip_d(X)$ is uniformly dense in $U_d(X)$ if and only if $X$ is {\it small-determined}. Recall that the class of the small-determined spaces were introduced by Garrido and Jaramillo in \cite{garrido2}, where it is proved  that, eventhough they are not all the metric spaces,  they form a big class  containing the normed spaces, the length spaces, or more generally  the so-called quasi-convex metric spaces. Hence, in the general frame, we have that  $H(U_d(X))$ and $H(Lip_d(X))$, that we will call  respectively the {\it Samuel realcompactification} and the {\it Lipschitz realcompactification} of the metric space $(X,d)$, could be different realcompactifications. In fact that will be the case for infinite bounded discrete metric spaces.

\smallskip

\section{The Lipschitz Realcompactification}

\smallskip

According to the above section, we already know several properties of $H(Lip_d(X))$. Namely, we can say that it is  the smallest realcompactification of the metric space $(X,d)$ where every function $f\in Lip_{d}(X)$ can be continuously  extended, it is contained in $H(Lip^*_d(X))=s_d X$, the Samuel compactification of $X$, and also that it can be described as, $$H(Lip_d(X))=\big\{\xi\in s_d X: \, f^*(\xi)\neq \infty \text{\,\, for all \,} f\in Lip_d(X)\big\}.$$

Our next result gives another characterization of $H(Lip_d(X))$ by using fewer Lipschitz functions. Namely, we will consider the family of functions $\{g_A: \emptyset\neq A\subset X\}$, where  $g_A:X\to \mathbb R$ is defined by $g_A(x)=d(x,A)=\inf\{d(x,a):a\in A\}$,  $x\in X$. Compare this  result with the analogous one obtained by Woods in \cite{woods} for the Samuel compactification.

\begin{proposition} Let $(X,d)$ be a metric space. Then, $H(Lip_d(X))$ is the smallest realcompactification of $X$ where, for every  $\emptyset\neq A\subset X$, the function $g_A$  can be continuously extended.

\begin{proof} Firstly, since every function $g_A$ is Lipschitz, then clearly it can be continuously extended to $H(Lip_d(X))$. On the other hand, if $Y$ is another realcompactification of $X$ with the above mentioned property,  we are going to see that $Y\geq H(Lip_d(X))$. Indeed, it is enough to check that every $f\in Lip_d(X)$ can also be continuously extended to $Y$. For that, we will use the extension result by Blair contained in  \cite{blair} saying that a real-valued continuous function $f$ defined on the dense subspace $X$ of $Y$ admits continuous extension to $Y$, if and only if, the two next conditions are fulfilled, where $\text{cl}_Y$ denotes the closure in the space $Y$,

\begin{enumerate}

\item If $a<b$, then $\text{cl}_{Y}\{x:\, f(x)\leq a\} \cap \text{cl}_{Y}\{x:\, f(x)\geq b\} = \emptyset$.

\item $\bigcap_{\,n\in \Nset}\text {cl}_{Y}\{x:\, |f(x)|\geq n\} = \emptyset$.

\end{enumerate}

\smallskip

So, let $f\in Lip_{d}(X)$, fix $x_0\in X$ and write,  for every $x\in X$,
$$f(x)=f(x_0) + \sup\{0, f(x)-f(x_0)\}-\sup\{0, f(x_0)-f(x)\}.$$
\noindent   We are going to apply the Blair result to the functions $h_1(x)=\sup\{0, f(x)-f(x_0)\}$ and $h_2(x)= \sup\{0, f(x_0)-f(x)\}$ in order to see that they, and hence $f$, can be continuously  extended to $Y$.

Since, clearly, $h_1$ is a Lipschitz function, there exists some constant $K\geq 0$, such that  $$0\leq h_1(x)\leq K\cdot d(x,x_0)=K\cdot g_{\{x_0\}}(x).$$
In particular, it follows that,  $$\bigcap_{\,n\in \Nset}\text{cl}_{Y}\big\{x:\, |h_1(x)|\geq n\big\} \subset \bigcap_{\,n\in \Nset}\text {cl}_{Y}\big\{x:\, |K\cdot g_{\{x_0\}} (x)|\geq n\big\}$$
\noindent and  since $K\cdot g_{\{x_0\}}$ satisfies above condition (2), then  $h_1$ also do.

On the other hand, let $a <b\in \Rset$ such that $\emptyset\neq A=\{x: h_1(x)\leq a\}$, and take the corresponding function $g_A$. Now, it is easy to check that $\{x:\, h_1(x)\leq a\}\subset \{x:\, g_A(x)\leq 0\}$ and $\{x:\, h_1(x)\geq b\}\subset \{x:\, g_A(x)\geq \frac{b-a}{K}\}$. Then, since  condition (1) is true for  $g_A$ then it is also true for $h_1$. Similarly, we can prove  that the function $h_2$ admits extension to $Y$, and this finishes the proof.
\end{proof}

\end{proposition}

As we have said before, the Lipschitz realcompactification of $X$ is a subspace of its Samuel compactification. Next result contains a useful description of this subspace. Recall that the analogous description as subspace of $\beta X$ can be seen in \cite{garrido1}.

\begin{proposition} \label{characterization-H(Lip_d(X))} Let $(X,d)$ be a metric space and $x_{0}\in X$. Then
$$H(Lip_d(X))=\bigcup_{n\in \Nset} {\rm cl}_{s_d X}B_{d}[x_{0}, n]$$
where $B_d[x_0, n]$ denotes the closed ball in $X$ around $x_0$ and radius $n$.
\begin{proof} Let $\xi \in \text{cl}_{s_d X}B_{d}[x_0, n]$, for some $n\in \Nset$. It is clear that every $f\in Lip_{d}(X)$ must be bounded in $B_{d}[x_{0}, n]$ and then its extension $f^{*}(\xi)\neq \infty$. Hence $\xi \in H(Lip_d(X))$.

Conversely, let $\xi \in H(Lip_d(X))$ and consider $f=d(\cdot ,x_{0})\in Lip_{d}(X)$. Since $f^{*}(\xi)\neq \infty$, we can choose $n>f^{*}(\xi)$. Then $\xi\in \text{cl}_{s_d X}B_{d}[x_0, n]$. Otherwise, there exists an open neighbourhood $V$ of $\xi$ in $s_d X$ such that $V\cap X\subset\{x: d(x,x_0)=f(x)> n\}$. Since $\xi \in \text{cl}_{s_d X}V=cl_{s_d X}(V\cap X)$, we have that $f^{*}(\xi)\geq n$, which is a contradiction.
\end{proof}
\end{proposition}

From the last representation of $H(Lip_d(X))$ we can derive the following result.

\begin{corollary} \label{Lipschitz realcompact. is locally compact} The Lipschitz realcompactification $H(Lip_d(X))$ of the metric space $(X,d)$ is a Lindel\"{o}f and locally compact topological space.
\end{corollary}
\begin{proof} Since $H(Lip_d(X))$ is $\sigma$-compact (i.e.,   countable union of compact sets) then it is Lindel\"{o}f. In order to see that it is also locally compact, let $\xi\in H(Lip_d(X))$, $x_0\in X $ and  $n\in \Nset$ such that $\xi \in \text{cl}_{s_d X}B_{d}[x_0, n]$. Now, since  the sets $B_{d}[x_0, n]$ and $X\setminus B_d[x_0, n+1]$ are at positive distance in $X$, then they have disjoint closures in $s_d X$ (see \cite{woods}). Then there is an open neighbourhood $V$ of $\xi$  in $s_dX$ such that $V\cap(X\setminus B_d[x_0, n+1])=\emptyset$. Therefore, $V\cap X\subset B[x_0, n+1]$, and then we have $$ \text{cl}_{s_d X}  V= \text{cl}_{s_d X} (V\cap X) \subset \text{cl}_{s_d X} B[x_0, n+1] \subset H(Lip_d(X)).$$
And we finish, since $\text{cl}_{s_d X}V$ is a compact neighbourhood of $\xi$ in $H(Lip_d(X))$.
\end{proof}
\begin{remark} We can deduce that  $H(Lip_d(X))$ is, in addition,   a hemicompact space. Recall that a topological space is said to be {\it hemicompact} if in the family of all its compact subspaces,  ordered by inclusion, there exists a countable cofinal subfamily. In this case we can see easily that for every compact $K\subset H(Lip_d(X))$ and every $x_0\in X$ there is $n\in \Nset$ such that $K\subset{\rm cl}_{s_d X}B_{d}[x_{0}, n]$.
\end{remark}

From the above topological results it is clear that  not every (realcompact) metric space is Lipschitz realcompact.  Recall that in this framework we say that a metric space is {\it Lipschitz realcompact} whenever $X=H(Lip_d(X))$ (see \cite{garrido1}). More precisely we can derive at once from Proposition \ref{characterization-H(Lip_d(X))} the following characterization.

\begin{proposition} \label{Lipschitz-realcompact} A metric space $(X,d)$ is Lipschitz realcompact if and only it satisfies the Heine-Borel property, i.e., every closed and $d$-bounded subset in $X$ is compact.
\end{proposition}

So, we can get different examples of metric spaces being or not Lipschitz realcompact. For instance, in the setting of Banach spaces to be Lipschitz realcompact is equivalent to have finite dimension. On the other hand,  note that Lipschitz realcompactness is not a uniform property in the frame of metric spaces. Indeed, from the above Proposition \ref{Lipschitz-realcompact}, the real line $\Rset$ is Lipschitz realcompact with the usual metric $d$ but not with the uniformly equivalent metric $\hat d=\min \{ 1, d\}$.

\smallskip

Then a natural question is when, for a metric space $(X,d)$, there exists an equivalent (resp. uniformly equivalent) metric $\rho$ such that $(X,\rho)$ is Lipschitz realcompact. In other words, we wonder when there exists an equivalent (resp. uniformly equivalent)  metric with the Heine-Borel property. These problems were studied respectively by Vaughan in \cite{vaughan} and  by Janos and Williamson in \cite{janos} (see also \cite{merono1}). From their results, we have the following fact (compare with last Corollary \ref{Lipschitz realcompact. is locally compact}).

\begin{corollary} For a metric space $(X,d)$ there exists a (uniformly) equivalent metric $\rho$ with $X=H(Lip_{\rho}(X))$ if and only if $X$ is Lindel\"{o}f and (uniformly) locally compact.
\end{corollary}

Now, if we  consider in the metric space $(X,d)$ a uniformly equivalent metric $\rho$, then it is clear that both metrics  will provide the same Samuel compactification, i.e. $s_d X\equiv s_\rho X$, since they define  the same  uniformly continuous functions on $X$. But, in general, there will be two different Lipschitz realcompactifications, namely $H(Lip_d(X))$ and $H(Lip_\rho(X))$. Taking into account that both realcompactifications are  contained in  the space $s_d X$, it is easy to see that, the order relation $H(Lip_d(X))\leq H(Lip_\rho(X))$ is equivalent to the inclusion relation $H(Lip_d(X))\supset H(Lip_\rho(X))$. Indeed, if $H(Lip_d(X))\leq H(Lip_\rho(X))$, there is some continuous function $h: H(Lip_\rho (X))\to H(Lip_d(X))\subset s_d X$ leaving $X$ pointwise fixed. Therefore, $h$ and the inclusion map $i: H(Lip_\rho (X))\to s_d X$ are two continuous functions that  coincide in the dense subspace $X$,  then $h=i$, that is $H(Lip_\rho (X))\subset  H(Lip_d(X))$.

\medskip

Next, let us see when these realcompactifications are comparable and also when they are equivalent.

\begin{proposition} \label{comparation of Lipschitz-realcomp.} Let $(X,d)$ be a metric space and $\rho$ a uniformly equivalent metric. The following are equivalent:
\begin{enumerate}
\item $H(Lip_d(X))\leq  H(Lip_\rho(X))$.
\item $H(Lip_d(X)) \supset H(Lip_\rho(X))$.
\item  If $B\subset X$ is $\rho$-bounded then it is $d$-bounded.
\end{enumerate}
\end{proposition}
\begin{proof} As we have said in the above paragraph conditions $(1)$ and $(2)$ are equivalent.

$(2) \Rightarrow (3)$ Let $B$  a $\rho$-bounded subset of $X$,  $x_0\in X$, and  $n\in \Nset$ such that $B\subset B_\rho[x_0, n]$. Then,
$$B\subset {\rm cl}_{s_\rho X} B_\rho[x_0, n] \subset H(Lip_\rho(X)) \subset H(Lip_d(X)).$$
Since ${\rm cl}_{s_\rho X} B_\rho[x_0, n]$ is a compact subspace of $H(Lip_d(X))$, then every real continuous function on $H(Lip_d(X))$ must be bounded on  $B$. Thus,  if $f^*$ is the continuous extension to $H(Lip_d(X))$ of the Lipschitz function $f=d(x_0,\cdot)$ then, that $f$ is bounded on $B$ means that  $B$ is $d$-bounded.

$(3) \Rightarrow (2)$ Let $x_0\in X$. This condition (3) says that for each $n\in \Nset$ there exists some $m_n\in \Nset$ such that $B_\rho[x_0, n]\subset B_d[x_0, m_n]$. And we finish by applying Proposition \ref{characterization-H(Lip_d(X))}  since, $$H(Lip_\rho(X))=\bigcup_{n\in \Nset} {\rm cl}_{s_\rho X}B_\rho [x_{0}, n] \subset \bigcup_{n\in \Nset} {\rm cl}_{s_d X}B_d[x_{0}, m_n]\subset H(Lip_d(X)).$$ \end{proof}

\begin{corollary} Let $(X,d)$ be a metric space and $\rho$ a uniformly equivalent metric. Then,   $H(Lip_d(X))$ and $H(Lip_\rho(X))$ are  equivalent realcompactifications of $X$
if and only if both metrics have the same bounded subsets (i.e., they are boundedly equivalent).
\end{corollary}

Recall that two metrics $d$ and  $\rho$ on $X$ are said to be {\it Lipschitz equivalent} when the identity maps $id:(X,d)\to (X,\rho)$ and $id:(X,\rho)\to (X,d)$ are Lipschitz. It is clear that Lipschitz equivalent metrics on $X$ provide the same bounded  subsets, the same Lipschitz functions and the same Lipschitz realcompactifications of $X$. On the other hand, as the next example shows, there exist uniformly equivalent metrics with the same bounded subsets, and hence (according to last result) with equivalent Lipschitz realcompactifications which are not Lipschitz equivalent.

\begin{example} \label{example 1} Consider on the real interval $X=[0,\infty)$ the usual metric $d$ and the metric $\rho$ defined by $\rho(x,y)=|\sqrt{x}-\sqrt{y}|$, for $x,y \in X$. Since the function $f(t)=\sqrt{t}$, for $t\geq 0$, is uniformly continuous but not Lipschitz with the usual metric,  we obtain that these metrics are uniformly equivalent but not Lipschitz equivalent. On the other hand, it is easy to see that both metrics have the same bounded sets, and therefore they give the same Lipschitz realcompactification.  Moreover, since $X$ is Heine-Borel with both metrics, then we know that in fact $X=H(Lip_{d}(X))=H(Lip_\rho(X))$.
\end{example}

We finish this section with a result (whose proof is immediate)  showing when the Lipschitz realcompactification of $X$ coincides with the Samuel compactification $s_d X$.

\smallskip
\begin{proposition} \label{Lipschitz realcompact. = Samuel compact.} Let $(X,d)$ a metric space. The following are equivalent:
\begin{enumerate}
\item $H(Lip_{d}(X))=s_d X$.
\item $Lip_d(X)=Lip^*_d(X)$
\item $X$ is $d$-bounded.
\end{enumerate}
\end{proposition}

\smallskip

\section{Some types of uniform boundedness}

\smallskip

As we have seen in the last section, the bounded subsets play an important role in our study. Now, we are going to recall two boundedness notions, preserved by uniformly continuous functions, that will be very useful later. We are referring to Bourbaki-boundedness and to $\alpha$-boundedness. These concepts  can be also defined in the general setting of uniform spaces but we are only interested here in their metric versions.

\smallskip

Let $(X,d)$ a metric space, and for $x_0\in X$ and $\varepsilon >0$,   define  $B^{1}_{d}(x_0, \varepsilon)=B_d(x_0, \varepsilon)$, where $B_d(x_0, \varepsilon)$ is the open ball of center $x_0$ and radius $\varepsilon$.  For every $m\geq 2$, let $$B^{m}_{d}(x_0, \varepsilon)=\bigcup \big\{B_{d}(x, \varepsilon):x\in B^{m-1}_{d}(x_0, \varepsilon)\big\}$$ and finally denote by
$B^{\infty}_{d}(x_0, \varepsilon)=\bigcup_{m\in \Nset}B^{m}_{d}(x_0, \varepsilon).$

\medskip

An $\varepsilon$-{\it chain} joining the points  $x$ and $y$ of $X$ is a finite sequence  $x=u_{0}, u_{1},..., u_{m}=y$ in $X$, such that  $d(u_{k-1},u_{k})<\varepsilon$, for $k=1,...,m$, where  $m$ indicates the length of the chain. We say that $x$ and $y$  are $\varepsilon$-\textit{chained} in $X$, when there exists some   $\varepsilon$-chain in $X$ joining  $x$ and $y$. Note that to be $\varepsilon$-chained defines an equivalence relation on $X$. Clearly this equivalence relation generates a  clopen uniform partition of $X$, where the equivalence class of each point $x$ is just the set $B^{\infty}_{d}(x, \varepsilon)$. These equivalences classes are called  $\varepsilon$-{\it chainable components} of $X$. Choosing in every $\varepsilon$-chainable component a representative point, say $x_i$, for $i$ running in a set $I_\varepsilon$ (which describes the number of $\varepsilon$ components), we can write:
$$X=\biguplus_{i\in I_{\varepsilon}}B^{\infty}_{d}(x_i, \varepsilon)$$
with  $B^{\infty}_{d}(x_i, \varepsilon)\cap B^{\infty}_{d}(x_j, \varepsilon)=\emptyset$, for    $i,j\in I_\varepsilon$, $i\neq j$, where the symbol $\biguplus$ denotes the disjoint union.

\begin{definition} A subset $B$ of a metric space $(X,d)$ is called $\alpha$-bounded in $X$ if for all $\varepsilon >0$ there exist finitely many points $x_1,...,x_n \in X$ such that $$B\subset \bigcup_{i=1}^{n} B^{\infty}_{d}(x_i,\varepsilon).$$ If $B=X$ then we will say that $X$ is an $\alpha$-bounded metric space.
\end{definition}

The notion of $\alpha$-boundedness was introduced and studied by Tashjian in the context of metric spaces in \cite{gloria1} and for  uniform spaces in \cite{gloria2}.

\begin{definition} A subset $B$ of a metric space $(X,d)$ is said to be Bourbaki-bounded in $X$ if for every $\varepsilon >0$ there exist $m\in \Nset$ and finitely many points $x_{1},...,x_{n}\in X$ such that $$B\subset \bigcup_{i=1}^{n} B^{m}_{d}(x_{i}, \varepsilon).$$ If $B=X$ then we say that $X$ is a Bourbaki-bounded metric space.
\end{definition}

Bourbaki-bounded subsets in metric spaces was firstly considered by Atsuji  under the name of {\it finitely-chainable} subsets (\cite{atsuji}). In the general context of uniform spaces we refer to the paper by Hejcman \cite{hejcman} where they are called {\it uniformly bounded} subsets.
Very recently Bourbaki-boundedness have been considered (with this precise name) in different frameworks, see for instance \cite{beer-garrido1}, \cite{beer-garrido2},  \cite{merono1}, \cite{merono2} and \cite{merono3}.

\smallskip

Note that if in the above definition, we have always $m=1$ (or even $m\leq m_0$ for some $m_0\in \Nset$), for every $\varepsilon >0$ then we get just total boundedness. It is important to point out here that for a subset to be Bourbaki-bounded or $\alpha$-bounded are not intrinsic properties, i.e., they depend on the ambient space where it is. For instance, if  $(\ell _2, \|\cdot\|)$ is the classical Hilbert space of all the real square summable sequences, and  $B=\{e_n: n\in \Nset\}$ is its  standard basis, then it is easy to check that $B$  is  Bourbaki-bounded and $\alpha$-bounded in $\ell _2$ but not in  itself.

\smallskip

On the other hand, it is clear that  $\alpha$-boundedness and Bourbaki-boundedness are uniform properties. And, in particular, uniformly equivalent metrics on a set $X$ provide the same $\alpha$-bounded and the same Bourbaki-bounded subsets. Note also that  connected metric spaces, as well as uniformly connected metric spaces (i.e., spaces that can not be the union of two sets at positive distance, as the rational numbers $\Qset$ with the usual metric), are  $\alpha$-bounded. On the other hand, every Bourbaki-bounded subset in the metric space $(X,d)$ is $\alpha$-bounded and also $d$-bounded,  but this last  does not occur with $\alpha$-boundedness. The real line $\Rset$ with the usual metric  is an example of an $\alpha$-bounded metric space which is not Bourbaki-bounded. In fact, in normed spaces, the Bourbaki-bounded subsets coincide with the bounded in the norm. Thus, if $B$ is the unit closed ball  of an infinite dimensional normed space, then $B$ is a Bourbaki-bounded subset which is not totally bounded. For further information related to these properties  we refer to \cite{merono1} and \cite{merono2}.

\smallskip

Now we are going to stress how  $\alpha$-boundedness and Bourbaki-boundedness are uniform boundedness notions.  For that, let   $f:X\rightarrow \Nset$ be a uniformly continuous  function,  where $\Nset$ is considered as a metric subspace of $\Rset$ with the usual metric, we will write $f\in U_{d}(X,\Nset)$. Then, if  $f\in U_{d}(X,\Nset)$ and $\varepsilon=1$, there must  exist  some $n\in \Nset$ such that $d(x,y)<1/n$ implies  $|f(x)-f(y)| < 1$. Clearly, this means that  $f$ will  be constant on each $1/n$-chainable component, and then the value of $f$ on each of these $1/n$-chains coincides with the value of $f$ on the corresponding representative point. Therefore, we can write $$U_{d}(X,\Nset)=\bigcup_{n\in \Nset} U_{d}^{n}(X,\Nset)$$  where   $U_{d}^{n}(X,\Nset)$, $n\in \Nset$,  denotes the family of functions defined as $$U_{d}^{n}(X,\Nset):=\Big\{f:\big\{x_i: \, i\in I_{1/n} \big\}\rightarrow \Nset\Big\}$$ where $I_{1/n}$ is, as we can say before,  the set of all the indexes describing the  $1/n$-chainable components of $(X,d)$.

\smallskip

From all the above, the following theorem given by Tashjian in  \cite{gloria1} is now clear.

\begin{proposition} \label{gloria-result} {\rm ({\sc Tashjian} \cite{gloria1}) } Let $B$ be a subset of a metric space $(X,d)$. The following statements are equivalent:
\begin{enumerate}
\item $B$ is $\alpha$-bounded in $X$.

\item Every function $f\in U_d(X, \Nset)$ is bounded on $B$.

\end{enumerate}
\end{proposition}

\smallskip

Next, by using the above notation, we are going to define a family of uniform equivalent metrics in the space $(X,d)$ that will permit not only  to characterize the Bourbaki-boundedness but also they will be very useful along the paper. The definition of these metrics are inspired by those considered by Hejcman in \cite{hejcman}.

\smallskip

For that, let  $n\in \Nset$, and consider all the $1/n$-chainable components of $X$. Next, on each  $1/n$-chainable component $B^{\infty}_{d}(x_i, 1/n)$ of $X$, $i\in I_{1/n}$, we can take the following metric $d_{1/n}$,  $$d_{1/n}(x,y)=\inf \sum_{k=1}^{m}d(u_{k-1},u_{k})$$
for $x,y\in B^{\infty}_{d}(x_i, 1/n)$, and  where the infimum is taken over all the $1/n$-chains, $x=u_0, u_1,..., u_m=y$ joining $x$ and $y$. Note that we may consider only those chains such that, if $m\geq 2$ then $d(u_{k-1},u_{k})+d(u_{k},u_{k+1})\geq 1/n$ since otherwise, due to the triangle inequality $d(u_{k-1},u_{k+1})\leq d(u_{k-1},u_{k})+d(u_{k},u_{k+1})<1/n$, the point $u_{k}$  can be removed from the initial chain. We will call these chains \textit{irreducible chains}.

\smallskip

It is easy to check that $d_{1/n}$ is a metric on each $1/n$-chainable component in a separately way, but we want to extend it to the whole space $X$. Now, for $f\in U^{n}_{d}(X,\Nset)$,  define  $\rho_{n,f}:X\times X\rightarrow [0,\infty)$ by,
$$\rho_{n,f}(x,y)= d_{1/n}(x,y) \text {\,\, if  \,\,}  x,y \in  B^{\infty}_{d}(x_i, 1/n), \, i\in I_{1/n}$$ and when $x\in  B^{\infty}_{d}(x_i, 1/n)$ and $y\in B^{\infty}_{d}(x_j, 1/n)$,  $i, j\in I_{1/n}$ with $i\neq j$, then
$$\rho_{n,f} (x,y)=d_{1/n}(x,x_i) + f(x_i)+ d_{1/n}(y, x_j)+f(x_j).$$

In order to check that $\rho _{n,f}$ is indeed a metric on $X$, we only  need to prove the triangle inequality for points $x, y, z$ which not all of them are  in the same $1/n$-chainable component. So, first take $x,y\in B^{\infty}_{d}(x_i, 1/n)$ and $z\in B^{\infty}_{d}(x_j, 1/n)$, with $i\neq j$, then
$$\rho _{n,f}(x,y)=d_{1/n}(x,y) \leq d_{1/n}(x, x_i)+ d_{1/n}(y, x_i)\leq$$
$$\leq d_{1/n}(x, x_i)+ d_{1/n}(y, x_i)+ 2\big[d_{1/n}(z, x_j)+f(x_i)+ f(x_j)\big]=$$
$$=\rho _{n,f}(x,z)+\rho _{n,f}(z,y).$$
Similarly when $x,z\in B^{\infty}_{d}(x_i, 1/n)$ and $y\in B^{\infty}_{d}(x_j, 1/n)$, with $i\neq j$, we have
$$\rho _{n,f} (x,y)=d_{1/n}(x, x_i)+ f(x_i) + d_{1/n} (y, x_j)+f(x_j)\leq $$
$$\leq d_{1/n}(x, z)+ d_{1/n}(z, x_i) +f(x_i) + d_{1/n} (y, x_j)+f(x_j)=$$
$$=\rho_{n,f}(x,z)+\rho_{n,f}(z,y).$$
And finally, suppose $x\in B^{\infty}_{d}(x_i, 1/n)$, $y\in B^{\infty}_{d}(x_j, 1/n)$ and $z\in B^{\infty}_{d}(x_k, 1/n)$, with $i\neq j\neq k\neq i$, then
$$\rho _{n,f} (x,y)=d_{1/n}(x,x_i)+ f(x_i) + d_{1/n}(y,x_j)+f(x_j)\leq $$
$$d_{1/n}(x,x_i)+ f(x_i) + d_{1/n}(y,x_j)+f(x_j)+ 2\big[d_{1/n}(z, x_k)+f(x_k)\big]=$$
$$=\rho _{n,f}(x,z)+\rho _{n,f}(z,y).$$

\medskip

And therefore $\rho _{n,f}$ is in fact a metric on $X$. On the other hand, note that $\rho_{n,f} (x,y)=d_{1/n}(x,y)=d(x,y)$ whenever $d(x,y)<1/n$ or $\rho_{n,f} (x,y)< 1/n$. That is, $d$ and $\rho_{n,f}$ are not only uniformly equivalent metrics but they are what is called  \textit{uniformly locally identical} (notion defined by Janos and Williamson in \cite{janos}). In particular, these metrics are also \textit{Lipschitz in the small equivalent}, i.e., the identity maps $id:(X,d)\rightarrow (X,\rho_{n,f})$ and $id:(X,\rho_{n,f})\rightarrow (X,d)$ are Lipschitz in the small. Recall that a function $f:(X,d)\rightarrow (Y,d\,')$ is said to be {\it Lipschitz in the small} if there exist $\delta>0$ and some $K>0$ such that $d\,'(f(x),f(y))\leq K\cdot d(x,y)$ whenever $d(x,y)<\delta.$ That is, $f$ is $K$-Lipschitz on every $d$-ball of radius $\delta$. This kind of uniform maps will play an important role in our study. For further information about these functions we refer to \cite{luukkainen} and \cite {garrido2}.

\smallskip

It must be pointed here that if we change the representative points in each  $1/n$-chainable component, that is, if  we choose another point $y_i \in B^{\infty}_{d}(x_i, 1/n)$, $i\in I_{1/n}$, and we define the corresponding metric $\omega_{n,f}: X \times X\rightarrow [0,\infty)$ similarly to $\rho _{n,f}$ but with the new representative points, then we still have that $d(x,y)=\omega_{n,f} (x,y)$ whenever $d(x,y)<1/n$ or $\omega_{n,f} (x,y)<1/n$. So that, the three metrics $\rho_{n,f}$, $\omega_{n,f}$ and $d$ are uniformly locally identical on $X$. Moreover, the election of these points will be irrelevant as we can see along the paper.

\smallskip

Next we are going to see   how the  family of metrics $\{\rho_{n, f}: {n\in \Nset, f\in U^{n}_{d}(X,\Nset)}\}$ are good for characterizing Bourbaki-boundedness (Proposition \ref{characterization-bourbaki-bounded} below). This characterization was essentially given  by  Hejcman in \cite{hejcman}, but we will split this result by means of  two technical lemmas which  will be very useful later. In the first lemma we describe the $\rho_{n,f}$-bounded subsets in $X$ for  $n\in \Nset$ and $f\in U^{n}_{d}(X,\Nset)$ fixed. And in the second one we do the same but fixing only  $n\in \Nset$, and varying all the functions $f$ in $U^{n}_{d}(X,\Nset)$.

\smallskip

\begin{lemma} \label{bounded n,f 1} Let $(X,d)$ be a metric space,  $n\in \Nset$ and $f\in U_{d}^{n}(X,\Nset)$. Then $B\subset X$ is $\rho _{n,f}$-bounded if  and only if there exist $F\subset I_{1/n}$ with  $f(\{x_i :i\in F\})$  finite, and  $M\in \Nset$ satisfying that $$B\subset \bigcup _{i \in F} B^{M}_{d}(x_ i , 1/n).$$ In particular,  every subset of the form $\bigcup _{i \in F} B^{M}_{d}(x_ i, 1/n)$,  for $F\subset I_{1/n}$ with   $f(\{x_i :i\in F\})$ finite, and $M\in\Nset$, is  $\rho _{n,f}$-bounded.
\end{lemma}
\begin{proof} Let $B\subset X$ be  $\rho _{n,f}$-bounded. Choose $i_0\in I_{1/n}$  and $R>0$ such that   $B\subset B_{\rho _{n,f}}(x_{i_0}, R)$.  From the definition of the metric $\rho _{n,f}$, it must exists  $F\subset I_{1/n}$ such that $f(\{x_i :i\in F\})\subset \Nset$ is finite and satisfying that $$B\subset \bigcup _{i \in F} B^{\infty}_{d}(x_ i , 1/n).$$  Take $K\in \Nset$ such that $f(x_i )\leq K$, for every $i\in F$. Now, if  $x\in B\cap B^{\infty}_{d}(x_ i , 1/n)$  we have that $$d_{1/n}(x, x_{i})=\rho_{n,f}(x, x_{i})\leq \rho_{n,f}(x, x_{i_0})+\rho_{n,f}(x_{i_0}, x_{i})<R+K+f(x_{i_0}),$$ and then there exists an irreducible $1/n$-chain in $B^{\infty}_{d}(x_ i, 1/n)$ joining $x$ and $x_{i}$, $x=u_0,u_1,...,u_m=x_{i}$ such that $\sum_{l=1}^{m} d(u_{l-1},u_{l})<R+K+f(x_{i_0})$.  Since the chain is irreducible, then if $m\geq 2$, every two consecutive sums satisfy $d(u_{l-1},u_{l})+d(u_{l},u_{l+1})\geq 1/n$, and then  $(1/n)(m-1)/2\leq R+K+f(x_{i_0})$. In particular, the length of every irreducible chain must be  less than $M$, where $M$ is a natural number with $M>2n(R+K+f(x_{i_0}))+1$. We finish, since we have that,
$$B\subset \bigcup_{i\in F}  B^{M}_{d}(x_ i , 1/n).$$

Conversely, let $M\in  \Nset$ and $F\subset I_{1/n}$ such that $f(\{x_i :i\in F\})$ is finite. We just need to prove that $\bigcup _{i \in F} B^{M}_{d}(x_ i , 1/n)$ is  $\rho _{n,f}$-bounded. So, take $K\in \Nset$ such that $f(x_i )\leq K$, for every $i\in F$. Fix some $j\in F$ and let $x\in  B^M_{d}(x_ i , 1/n)$, $i\in F$.  Then $$\rho _{n,f}(x, x_j)\leq \rho _{n,f}(x, x_i)+\rho _{n,f}(x_i,x_j)$$ $$\leq d_{1/n}(x, x_i)+f(x_i)+f(x_j)\leq M/n + 2K$$
that is, $x\in B_{\rho_{n,f}}(x_j, R)$, for $R>M/n + 2K$. We finish since $\bigcup _{i \in F} B^{M}_{d}(x_ i , 1/n)\subset  B_{\rho_{n,f}}(x_j, R)$.

\end{proof}

\begin{lemma} \label{bounded n,f 2} Let $(X,d)$ be a metric space and $n\in \Nset$. Then $B\subset X$ is $\rho _{n,f}$-bounded, for every $f\in U_{d}^{n}(X,\Nset)$, if and only if there exist a finite set $F\subset I_{1/n}$  and  $M\in \Nset$ satisfying that $$B\subset \bigcup _{i \in F} B^{M}_{d}(x_ i , 1/n).$$
\end{lemma}
\begin{proof} According to the above Lemma \ref{bounded n,f 1}, only we need to prove that if  $B$ is  $\rho _{n,f}$-bounded for every $f\in U_{d}^{n}(X,\Nset)$, then a finite set $F$ in $I_{1/n}$ can be taken satisfying the statement. Indeed,  note that  $B$ only meets finitely many $1/n$-chainable components, otherwise we can choose some  $f\in U^{n}_{d}(X,\Nset)$ such that $f(B)$ is an infinite subset of $\Nset$, and then (by the definition of $\rho_{n,f}$) $B$ will be not $\rho _{n,f}$-bounded, which is  a contradiction. Therefore we finish by  taking $F$ the finite subset of $I_{1/n}$ of these $1/n$-chainable components. \end{proof}

Now we are ready to establish the announced characterization  of Bourbaki-boundedness by means of the family of metrics $\{\rho_{n,f}: {n\in \Nset, f\in U^{n}_{d}(X,\Nset)}\}$. Note that equivalences of (1) and (2) in the next two propositions were given by Hejcman in \cite{hejcman}.

\begin{proposition} \label{characterization-bourbaki-bounded} Let $(X,d)$ be a metric space and $B\subset X$. The following are equivalent:
\begin{enumerate}
\item $B$ is Bourbaki-bounded in $X$.

\item $B$ is $\rho$-bounded for every uniformly equivalent metric $\rho$.

\item $B$ is $\rho$-bounded for every Lipschitz in the small equivalent metric $\rho$.

\item $B$ is $\rho _{n,f}$-bounded for every  $n\in \Nset$ and $f\in U^{n}_{d}(X,\Nset)$.
\end{enumerate}
\end{proposition}
\begin{proof} All the (ordered) implications follow at once, except $(4)\Rightarrow (1)$. So, suppose $B$ satisfies $(4)$ and let $\varepsilon>0$. Take $n\in \Nset$, such that $1/n<\varepsilon$. By the previous Lemma \ref{bounded n,f 2} we have that, for some finite set $F\subset I_{1/n}$ and $M\in \Nset$,
$$B\subset  \bigcup _{i \in F} B^{M}_{d}(x_ i , 1/n)\subset\bigcup _{i \in F} B^{M}_{d}(x_ i, \varepsilon)$$
and then, $B$ is Bourbaki-bounded.
\end{proof}

We finish this section with an analogous result to the above Proposition \ref{gloria-result}, for Bourbaki-bounded subsets. It can be also seen in \cite{beer-garrido1} but with a different proof. Recall that $LS_d(X)$ denoted the family of all the real-valued functions on $X$ that are Lipschitz in the small.

\begin{proposition} \label{Bourbaki-bounded versus uniform functions} Let $(X,d)$ be a metric space and $B\subset X$. The following are equivalent:
\begin{enumerate}
\item $B$ is Bourbaki-bounded in $X$.

\item Every function $f\in U_{d}(X)$ is bounded on $B$.

\item Every function $f\in LS_{d}(X)$ is bounded on $B$.
\end{enumerate}
\begin{proof}$(1)\Rightarrow (2)$ If $f\in U_{d}(X)$ is not bounded on $B$, then the metric $\rho(x,y)=d(x,y)+|f(x)-f(y)|$ is uniformly equivalent to $d$ but $B$ is not $\rho$-bounded. By Proposition \ref{characterization-bourbaki-bounded},  $B$ is not Bourbaki-bounded.

$(2)\Rightarrow (3)$ is clear.

$(3)\Rightarrow (1)$ If $B$ is not Bourbaki-bounded then, again by Proposition \ref{characterization-bourbaki-bounded}, there exists some metric $\rho$ Lipschitz in the small equivalent to $d$ such that $B$ is not $\rho$-bounded. Then, taking some $x_0\in X$ and the function $f(x)=\rho(x,x_0)$, $x\in X$, we have that $f\in LS_{d}(X)$ is not bounded in  $B$.
\end{proof}
\end{proposition}

\smallskip

\section{The Samuel Realcompactification}

\smallskip

This section is devoted to study the so-called  Samuel realcompactification $H(U_d(X))$ of a metric space $(X,d)$. That is, the smallest realcompactification of $X$ with the property that every real-valued uniformly continuous function can be  continuously extended to it.
Note that this realcompactification can be considered as the completion of the so-called {\it c-modification} of $X$, i.e., the space $X$ endowed with the weak uniformity generated by the real-valued uniformly continuous functions (see \cite{husek}).

\smallskip

We know that this realcompactification is contained in the corresponding Samuel compactification, and more precisely we have that,
$$H(U_d(X))=\big\{\xi\in s_d X: \, f^*(\xi)\neq \infty \text{\,\, for all \,} f\in U_d(X)\big\}.$$
Moreover, according to the above sections, we have the  following: $$ X\subset H(U_d(X))\subset s_d X$$

\noindent where the reverse inclusions are not necessarily true, as the next examples show.
\smallskip

\begin{enumerate}

\item [1.] Let $X$ be the closed unit ball of an infinite dimensional Banach space. Since every uniformly continuous function on $X$ is bounded, then $H(U_d(X))=H(U^*_d(X))=s_d X$. But, as $X$ is not compact, then $X\neq H(U_d(X))$. Moreover if $X$ has non-measurable cardinal then we know that $X$ is a realcompact space, and that means that $X=\upsilon X \neq H(U_d(X))$. Then, this example illustrates how the Samuel  and the Hewitt-Nachbin realcompactifications may differ.
\smallskip

\item [2.] Since the real line $\Rset$ with the usual metric admits unbounded real-valued uniformly continuous functions, then $H(U_d(\Rset))$ can not be compact, and then $H(U_d(\Rset))\neq s_d \Rset$.
\end{enumerate}

\smallskip

The problem of when the Samuel realcompactification and the Samuel compactification of a metric space coincides  has the following easy answer (compare with Proposition \ref{Lipschitz realcompact. = Samuel compact.}).

\begin{proposition} \label{Samuel realcompactification = Samuel compactification} Let $(X,d)$ a metric space. The following are equivalent:
\begin{enumerate}
\item $H(U_{d}(X))=s_d X$.
\item $U_d(X)=U^*_d(X)$
\item $X$ is Bourbaki-bounded.
\end{enumerate}
\end{proposition}

On the other hand,  to know  when $H(U_d(X))$ is just $X$, i.e., when $X$ is what we will call {\it Samuel realcompact} is not so easy as in the case of Lipschitz-realcompactness and it deserves to be studied in a separate section. So,  we will devoted Section 5 to this end.

\smallskip

Concerning the relationship between  $H(Lip_d(X))$ and $H(U_d(X))$ we can say that they are always comparable  realcompactifications of   $X$. Indeed, it is clear that $H(U_d(X))\geq H(Lip_d(X))$. And therefore, since they live in the same compactification $s_dX$, we have that this relation turns into $$H(U_d(X))\subset H(Lip_d(X)).$$
Next example shows that  the reverse inclusion does not occur.

\begin{example} Let $X$ be an infinite space endowed  with the $0-1$ discrete metric $d$. Then $Lip_{d}(X)=C^{*}(X)$ and  $U_{d}(X)=C(X)$. Therefore  $H(Lip_{d}(X))=\beta X$ and $H(U_{d}(X))=\upsilon X$. Now, since $X$ is  not compact, then we have that $H(U_d(X))\neq H(Lip_d(X))$.
\end{example}

\smallskip

Next, if we take another metric $\rho$ uniformly equivalent to $d$,  we have that $$H(U_d(X))=H(U_\rho(X))\subset H(Lip_\rho(X)).$$ That is, the Samuel realcompactification of $(X,d)$ is  contained in every Lipschitz realcompactification of $X$ given by a uniform equivalent metric.  Surprisingly, next result says  that $H(U_d(X))$ is in fact the supremum, with the usual order in the family of  the realcompactifications, of all these Lipschitz realcompactifications. Note that the supremum  of a family of realcompactifications of $X$ which are contained in the same space, always exists, and it is actually nothing more than the intersection of all of them (which is also a realcompactification of $X$).

\begin{theorem} \label{Samuel as supremum Lipschitz} Let $(X,d)$ be a metric space. Denoting ``uniformly equivalent'' by  $\stackrel{{u}}{\sim}$, then
$$H(U_d (X))=\bigvee_{\rho\stackrel{{u}}{\sim} d} H(Lip_\rho(X))=\bigcap_{\rho\stackrel{{u}}{\sim} d} H(Lip_\rho(X)).$$

\begin{proof} As we have previously pointed out, $H(U_d(X))=H(U_\rho(X))\subset H(Lip_\rho(X))$, for every metric $\rho$ uniformly equivalent to $d$. And then that $H(U_d (X))$ is contained in $\bigcap_{\rho\stackrel{{u}}{\sim} d} H(Lip_\rho(X)$ is clear.

For the reverse inclusion, let $f\in U_{d}(X)$ and $\rho$ the metric defined as $\rho(x,y)=d(x,y)+|f(x)-f(y)|$, $x,y\in X$. Then $\rho$ is uniformly equivalent to $d$ and clearly  $f\in Lip_{\rho}(X)$. Since  $f$ can be continuously extended to $H(Lip_\rho(X))$ then it can be also extended to the realcompactification $\bigcap_{\rho\stackrel{{u}}{\sim} d} H(Lip_\rho(X))$. Finally,  as $H(U_d(X))$ is the smallest realcompactification with this property, then it follows that $\bigcap_{\rho\stackrel{{u}}{\sim} d} H(Lip_\rho(X))\subset H(U_d(X))$, as we wanted.\end{proof}
\end{theorem}

As a consequence of the precedent theorem, we can see when  the Samuel and the Lipschitz realcompactifications coincide for a metric space.

\begin{proposition} \label{Samuel real. = Lipschitz real.} Let $(X,d)$ a metric space. Then $H(U_d(X))=H(Lip_d(X))$ if and only if every real-valued uniformly continuous function on $X$ is bounded on every $d$-bounded set.
\begin{proof} Suppose $H(U_d(X))=H(Lip_d(X))$, and take $x_0\in X$. If $B$ is a $d$-bounded subset of $X$, then there exists some $N\in \Nset$, such that $B\subset B_d[x_0, N]$. Now if $f\in U_d(X)$, then $f$ admits continuous extension to $H(U_d(X))=H(Lip_d(X))=\bigcup_{n\in \Nset} {\rm cl}_{s_d X}B_d[x_{0}, n]$ (Proposition \ref{characterization-H(Lip_d(X))}). Therefore the extension function $f^*$ must be bounded on the compact space ${\rm cl}_{s_d X}B_{d}[x_{0}, N]$, and this means in particular that $f$ is bounded on $B$.

Conversely,  from the above Theorem \ref{Samuel as supremum Lipschitz},  it is enough to check that $H(Lip_d(X))\subset H(Lip_\rho(X))$, for any $\rho$ uniformly equivalent to $d$. But, this is equivalently to see, according to Proposition \ref{comparation of Lipschitz-realcomp.}, that every $d$-bounded set is $\rho$-bounded. So, let $B$ be a $d$-bounded set in $X$, since the function $f(\cdot)=\rho(\cdot, x_0)$ is uniformly continuous, then from the hypothesis we have that $f$  must be bounded on $B$, and this means clearly that $B$ is $\rho$-bounded, as we wanted.
\end{proof}
\end{proposition}

\smallskip

It is clear that last result can be reformulated in terms of Bourbaki-bounded subsets. Indeed,  from Proposition \ref{characterization-bourbaki-bounded}, we can say that the Samuel  and the Lipschitz realcompactifications of the metric space $(X,d)$ are equivalent if and only if every bounded set in $X$ is Bourbaki-bounded, or equivalently, in terms of bornologies, when the bornology of the bounded sets, denoted by $\textbf{B}_{d}(X)$,  and the bornology of the Bourbaki-bounded sets, denoted by $\textbf{BB}_{d}(X)$,  coincide. Recall that, as we have said in the Introduction,  a family of subsets of $X$ is a {\it bornology} whenever they forms a cover of $X$, closed by finite unions, and stable by subsets.

\smallskip

Thus,  we can describe  the Lipschitz and the Samuel realcompactifications in terms  of bornologies. Indeed,  Proposition \ref{characterization-H(Lip_d(X))} and Theorem \ref{Samuel as supremum Lipschitz}  can be written respectively as follows, $$H(Lip_d(X))=\bigcup_{B\in \textbf{B}_{d}(X)} {\rm cl}_{s_d X}B  \text{\,\,\,\,\, and \,\,\,\,\,} H(U_d(X))=\bigcap_{\rho\stackrel{{u}}{\sim} d} \,\, \bigcup_{B\in \textbf{B}_{\rho}(X)} {\rm cl}_{s_d X} B.$$

Note that, according to the characterization of Bourbaki-bounded subsets given in Proposition \ref{characterization-bourbaki-bounded}, it is natural to wonder if in the last equality  we can replace  ``$\bigcap_{\rho\stackrel{{u}}{\sim} d} \,\, \bigcup_{B\in \textbf{B}_{\rho}(X)}$'' by just the symbol ``$\bigcup_{B\in \textbf{BB}_{d}(X)}$''. That is, we wonder if it is true that $H(U_d(X))= \bigcup_{B\in \textbf{BB}_{d}(X)} {\rm cl}_{s_d X}  B$.  On one hand, it is clear that if $\xi\in {\rm cl}_{s_d X} B$, for some Bourbaki-bounded subset $B$, then $\xi\in H(U_d(X))$. That is
$$H(U_d(X))\supset \bigcup_{B\in \textbf{BB}_{d}(X)} {\rm cl}_{s_d X}  B.$$
But, unfortunately, the reverse inclusion is not true. For instance, if $X$ is a set with a measurable cardinal endowed with the $0-1$ metric then the  Bourbaki-bounded subsets are only the finite ones. And then $H(U_d(X))=H(C(X))=\upsilon X\neq X=\bigcup_{B\in \textbf{BB}_{d}(X)} {\rm cl}_{s_d X} B.$

\smallskip

Last example shows that a  description of the Samuel realcompactification as the above obtained for $H(Lip_d(X))$  by means of the $d$-bounded sets in $X$, that are precisely those sets where the Lispchitz functions are bounded, seems to be not true. In other words, we do not know whether $H(U_d(X))$ is  what we could call  a {\it (uniformly) bornological  realcompactification}. Recall that in the setting of topological Tychonoff spaces, the notion of bornological realcompactification was given by  Vroegrijk in \cite{vroegrijk}. There, he   study those realcompactifications of a topological space $X$,  which are contained in $\beta X$, and  that in some sense  are given by bornologies associated to families of continuous functions.  We think that it would be very interesting to make an analogous study in the setting of metric spaces together with bornologies defined by families of uniformly continuous functions. But, we will not do this in the present  paper, nevertheless we  will devote next Section 7 to see that by relating bornologies  and realcompactifications we can obtain  a deeper knowledge of some of our main objectives here.

\smallskip

 We finish this section giving another description of the  space $H(U_d(X))$ that will be useful later. This description  is in  the line of Theorem \ref{Samuel as supremum Lipschitz} but with less uniformly equivalent metrics, namely with  the family of  metrics before  defined  $\{\rho_{n,f}\}$. Recall  that (as we have said before) two metrics $d$ and  $\rho$ are Lipschitz in the small equivalent whenever the identity maps $id:(X,d)\rightarrow (X,\rho)$ and $id:(X,\rho)\rightarrow (X,d)$ are Lipschitz in the small.

\begin{theorem}  \label{Samuel as supremum Lipschitz-2} Let $(X,d)$ be a metric space. Denoting  ``Lipschitz in the small  equivalent'' by $\,\stackrel{{LS}}{\sim}$, then
$$H(U_d (X))=\bigcap  \big \{H(Lip_\rho(X)): \rho \stackrel{{LS}}{\sim} d\big \}=\bigcap \big\{ H(Lip_{\rho_{n,f}}(X)): {n\in \Nset,\, f\in U_d^n(X,\Nset)}\big\}.$$
\begin{proof} According to all the precedent study, it is only necessary to check that the last space is contained in the first one. Indeed, let $\xi \in \bigcap \big\{ H(Lip_{\rho_{n,f}}(X)): {n\in \Nset,\, f\in U_d^n(X,\Nset)}\big\}$. In order to prove that $\xi$ belongs to $H(U_d (X))$ we are going to see  that for every $g\in U_d(X)$ we have that $g^*(\xi)\neq\infty$.  Firstly recall that the family of Lipschitz in the small functions are uniformly dense in $U_d(X)$ (see \cite{garrido2}). Then we can suppose that $g\in LS_d(X)$ and, without loss of generality,  that $g\geq 1$. Thus, let  $K>0$ and $n\in\Nset$ be, such that $|g(x)-g(y)|\leq K\cdot d(x,y)$, whenever $d(x,y) <1/n.$

Take the metric $\rho_{n,f}$ with $f:X\to \Nset$ defined   by $f(x)= [g(x_i)] +1$ (where $[t]$ denotes the integer part of  $t$),  for $x\in B^{\infty}_{d}(x_i, 1/n)$ and $i \in I_{1/n}$. Since $\xi\in H(Lip_{\rho_{n,f}} (X))$ then $\xi\in {\rm cl}_{s_dX} B$, for some $\rho_{n,f}$-bounded subset $B$ of $X$ (Proposition \ref{characterization-H(Lip_d(X))}). Now, from Lemma \ref{bounded n,f 1},  there exist $F\subset I_{1/n}$ with  $f(\{x_i :i\in F\})$  finite, and  $M\in \Nset$ satisfying that $$\xi \in {\rm cl}_{s_dX}  \bigcup _{i \in F} B^{M}_{d}(x_ i , 1/n).$$
Finally, it is enough to check that $g$ is bounded on $\bigcup _{i \in F} B^{M}_{d}(x_ i , 1/n)$. Indeed, let $L\in \Nset$ such that $f(x_i)\leq L$, for $i\in F$. Next, fix $j\in F$ and let $x\in B^{M}_{d}(x_ i , 1/n)$ for some $i\in F$. Then,
$$|g(x)-g(x_{j} )|\leq |g(x)-g(x_{i} )|+|g(x_{i} )-g(x_{j} )|\leq K\cdot M/n +2L.$$
Therefore,  $g$ is bounded on  $\bigcup_{i\in F} B^{m}_{d}(x_{i} ,1/n)$, and then $g^*(\xi)\neq \infty$, as we wanted.
\end{proof}
\end{theorem}

\smallskip

\section{Samuel realcompact metric spaces}

\smallskip

In this section we are going to obtain an analogous result to the well known Kat\v{e}tov-Shirota theorem.  Recall that in this theorem it was  characterized the realcompactness of complete uniform spaces  by means of the non-measurability of the cardinality of its  closed discrete subspaces (see \cite{gillman}). We will give an analogous result in the special case of metric spaces and for the Samuel realcompactness. As we have said before, a metric space $(X,d)$ is said to be  Samuel realcompact whenever $X=H(U_d(X))$. First of all, note that if $X$ is Samuel realcompact then it is in particular realcompact since its coincides with one of its realcompactifications.

\smallskip

If we think the Samuel realcompactification of a uniform space as the completion of its $c$-modification (see the first paragraph in Section 4), then a space is Samuel realcompact whenever its  $c$-modification is complete. This is precisely the definition of {\it uniform realcompleteness} given by Hu\v{s}ek and Pulgar\'{\i}n in \cite{husek}. We must also add here that N{\aa}jstad \cite{najstad} defined for a proximity space $X$ to be \textit{realcomplete} if for every $\xi \in s_{\mu} X-X$ there exists a proximally-continuous mapping that cannot be (continuously) extended to $\xi$, where $s_{\mu} X$ is the Samuel compactification of $X$ together with the totally bounded uniformity $\mu$ generated by the proximity (see \cite{alfsen}). In particular, since for a metric space, proximally-continuous functions for the metric proximity are exactly the uniformly continuous functions (see \cite{naimpally}), then  the notion of realcompleteness by N{\aa}jstad coincides with  Samuel realcompactness, at least in the frame of metric spaces.

\smallskip

On the other hand, examples of Samuel realcompact spaces are, for instance, every finite dimensional normed space, or more generally every Lipschitz realcompact metric space. Moreover, every uniformly discrete metric space  of non-measurable cardinal is also Samuel realcompact since it is in particular realcompact and $C(X)=U_d(X)$. Recall that a metric space $(X,d)$ is said to be {\it uniformly discrete} when there exists some $\varepsilon >0$, such that $d(x,y)>\varepsilon$, for $x\neq y$. For this class of metric spaces we have the following result characterizing its Samuel realcompactness, that will be also useful to derive our general result. In order to establish this, we need to  recall the notion of $\alpha$-bounded filter. A  filter $\mathfrak{F}$ in a  metric space $(X, d)$ is said to be $\alpha$-{\it bounded} when  every $f\in U_{d}(X,\Nset)$ is bounded (finite) on some  member $F\in \mathfrak{F}$. Note that, the equivalence between (2) and (4) in the next result could be known, nevertheless we provide a complete proof using  the special description of the Samuel realcompactification of any uniformly discrete space.

\begin{theorem} \label{Samuel-realcompactness for unif. discrete} Let $(X,d)$ a uniformly discrete metric space. The following are equivalent,
\begin{enumerate}
\item $X$ has  non-measurable cardinal.

\item $X$ is realcompact.

\item $X$ is Samuel realcompact.

\item Every $\alpha$-bounded ultrafilter  in $X$  is fixed.
\end{enumerate}
\begin{proof} The equivalence between $(1)$ and $(2)$ is well known (see \cite{gillman}). On the other hand, that $(2)$ and $(3)$ are equivalent follows at once since in this case $C(X)=U_d(X)$, and therefore  $\upsilon X=H(U_d(X))$.

Let us prove the equivalence between $(3)$ and $(4)$. First note that for a uniformly discrete metric space we have that its Samuel realcompactification can be also described as follows: $$H(U_d(X))=\big\{\xi\in s_d X: \, f^*(\xi)\neq \infty \text{\,\, for all \,} f\in U_d(X, \Nset)\big\}.$$ Indeed, it is clear that the first set is contained in the second one. To the reverse  inclusion, suppose that there exists $\xi \in s_d X$ such that $f^{*}(\xi)\neq\infty$, for all $f\in U_{d}(X,\Nset)$,   but for some $g\in U_{d}(X)$ we have $g^{*}(\xi)=\infty$. Note that with loss of generality, we can suppose $g\geq 1$. Consider $h:X\rightarrow \Nset$ the integer part of $g$, that is $h(x)=[g(x)]$, $x\in X$.  Clearly  $h\in U_{d}(X,\Nset)$, and $h^*(\xi)=\infty$ which is a contradiction.

So, to see that $(3)\Rightarrow (4)$, suppose $X$ is Samuel realcompact, and let $\mathfrak{F}$ an ultrafilter such that for every $f\in U_d(X, \Nset)$ there is $F_f\in\mathfrak{F}$ with  $f$  bounded on $F_f$. By compactness, let $\xi \in \bigcap \{{\rm cl}_{s_d X} F: F\in \mathfrak{F}\}$. Then, for every $f\in U_d(X, \Nset)$, $$f^{*}(\xi)\in f^{*}({\rm cl}_{s_d X} F_f)\subset \overline{f(F_f )}$$ and clearly $f^*(\xi)\neq \infty$. So, $\xi\in H(U_d(X))=X$ and $\mathfrak{F}$ must be  fixed since $$\xi \in \bigcap \big\{{\rm cl}_{s_d X} F \cap X: F\in \mathfrak{F}\big\}=\bigcap \big\{F: F\in \mathfrak{F}\big\}.$$

$(4)\Rightarrow (3)$ Now, suppose  $\xi \in H(U_d(X))$. Then, according to the above description of the Samuel realcompactification, we have that  $f^{*}(\xi)\neq\infty$, for all  $f\in U_d(X,\Nset)$. Thus, for every $f\in U_d (X,\Nset)$ there exists $V_f$ a neighbourhood of $\xi$ in $s_d X$ such that $f$ is bounded on $V_f \cap X$. Let $\mathfrak{F}$ an ultrafilter in $X$ containing the filter $\{V\cap X: V \text{ neighbourhood of } \xi \text{ in\,} s_d X\}$. In particular,  $\mathfrak{F}$ is $\alpha$-bounded and so it must be fixed in $X$. But $$\emptyset\neq \bigcap \big\{F: F\in \mathfrak{F}\big\} \subset \bigcap \big\{V\cap X: V \text{ neighbourhood of } \xi \text{ in\,} s_d X\big\} \subset \{\xi\}\cap X$$ so $\xi \in X$, as we wanted.
\end{proof}
\end{theorem}

Next example shows that the above result does not work for discrete metric spaces that are not uniformly discrete, and in particular that realcompactness is not equivalent to Samuel realcompactness, even for discrete spaces.

\begin{example} \label{example 2} Consider $X=\{1/n: n\in \Nset\}$ with its usual metric $d$. Then $(X,d)$ is a discrete metric with countable (and so non-measurable) cardinal that is realcompact but not Samuel realcompact. Indeed, every uniformly continuous function on $X$ can be continuously extended to its (compact) completion $\widetilde{X}=\{1/n: n\in \Nset\} \cup \{0\}$, and then it must be bounded. Then, we have that $X=\upsilon X\neq  H(U_d(X))=s_d X$. Moreover, since $X$ is a totally bounded  metric space,  then we know that its Samuel compactification is in fact its metric completion $\widetilde{X}$ (see  \cite{woods}).
\end{example}

\smallskip

We are going to state our main result in this section that can be considered in the line of the well known Kat\v{e}tov-Shirota result. Here,   we will see that for metric spaces (with some additional non-measurable cardinal property) Samuel realcompactness is equivalent to some kind of completeness, namely Bourbaki-completeness. So, we need to recall the notion of Bourbaki-complete metric space  introduced and studied in \cite{merono2}.

\begin{definition} {\rm (\cite{merono2})}  A metric space $(X,d)$ is said to be {\it Bourbaki-complete} if every Bourbaki-Cauchy net clusters, where a net $(x_{\lambda})_{\lambda\in \Lambda}$ is {\it Bourbaki-Cauchy} if for every $\varepsilon>0$ there exist $m\in \Nset$ and $\lambda_{0} \in \Lambda$ such that for some $x_0\in X$, $x_{\lambda}\in B^{m}_d(x_0, \varepsilon)$, for all $\lambda \geq \lambda _{0}$.
\end{definition}

Note that every compact metric space is Bourbaki-complete and that every Bourbaki-complete metric space is complete (every Cauchy net is in particular Bourbaki-Cauchy). However the reverse implications are not true. For instance,  $\Rset$ with the usual metric is a Bourbaki-complete metric space which is not compact and every infinite dimensional Banach space is complete but not Bourbaki-complete (see next Proposition). More examples and results  of Bourbaki-completeness are  given in \cite{merono2}. For instance,  it can be seen there that the notion of Bourbaki-completeness by nets is equivalent to the corresponding notion by sequences, as the same happens with the usual completeness. For this reason we will use  either nets or sequences as convenient.

\smallskip

On the other hand, it is interesting to realize that the role that the Bourbaki-bounded subsets play into the  Bourbaki-complete metric spaces is the same as the totally bounded subsets  play for the usual completeness. Next result, that we will use later, makes it clear.

\begin{proposition} {\rm (\cite{merono2})} \label{Bourbaki-complete and Bourbabi-bounded} For  a metric space $(X,d)$ the following are equivalent:
\begin{enumerate}
\item $X$ is Bourbaki-complete.

\item Every closed Bourbaki-bounded subset of $X$ is compact.
\end{enumerate}
\end{proposition}

Now, we have all the ingredients to establish our main result about Samuel realcompactness. But previously we need the following Lemma  whose proof follows using the above characterization of Samuel realcompactness for uniformly discrete spaces (Theorem \ref{Samuel-realcompactness for unif. discrete}).

\begin{lemma} \label{Non-measurable cardinal I_1/n} Let $(X,d)$ be a metric space and  $n\in \Nset$. If $I_{1/n}$ has non-measurable cardinal then $$ H(U_d(X))\subset \biguplus_{i\in I_{1/n}} {\rm cl} _{s_{d} X } B^{\infty}_{d}(x_{i},1/n).$$
\begin{proof} First note that always  the union appearing in the above formula is a disjoint union, since $B^{\infty}_{d}(x_{i},1/n)$ and $B^{\infty}_{d}(x_{j},1/n))$, $i\neq j$, are subsets in $X$ that are $1/n$ $d$-apart and then they have disjoint closure in $s_d X$ (see \cite{woods}).

Now, let $\xi\in H(U_d(X))$. If we consider all the $1/n$-chainable components of $X$,  these components are a uniform partition of $X$. Then, the subspace  formed by all the representative points, that we can identify with $I_{1/n}$, is uniformly discrete. Since $I_{1/n}$  has  non-measurable cardinal, then it satisfies all the equivalent conditions  in Theorem \ref{Samuel-realcompactness for unif. discrete}. Now we are going to consider  an special $\alpha$-bounded ultrafilter $\mathfrak{F}$ in $I_{1/n}$. Namely, let
$$\mathfrak{F}=\Big\{F\subset I_{1/n}:  \xi \in {\rm cl}_{s_d X} \Big (\bigcup_{i \in F}  B^{\infty}_{d}(x_i,1/n)\Big )\Big\}.$$

First, note that $\mathfrak{F}\neq \emptyset$, since $F=I_{1/n}$ clearly belongs to  $\mathfrak{F}$. Moreover,  as we have noticed previously, two sets that are $1/n$ $d$-apart in $X$  have disjoint closures in $s_dX$. Thus,  if $F, F\,' \in \mathfrak{F}$, then $F\cap F\,'\neq \emptyset$, since (by definition of $\mathfrak{F}$)
$$\xi \in  {\rm cl}_{s_d X} \Big(\bigcup_{i \in F}B^{\infty}_{d}(x_i,1/n)\Big)\bigcap  {\rm cl}_{s_d X} \Big(\bigcup_{i \in F\,'}B^{\infty}_{d}(x_i,1/n)\Big).$$
Furthermore, we have that $F\cap F\,'\in \mathfrak{F}$. Indeed, taking  into account that  $F=(F\cap F\,')\cup (F\setminus F\,')$,   $F\,'=(F\cap F\,')\cup (F\,'\setminus F)$ and also that the sets $\Big(\bigcup_{i \in F\setminus F\,'}B^{\infty}_{d}(x_i,1/n)\Big)$ and $\Big(\bigcup_{i \in F\,'\setminus F}B^{\infty}_{d}(x_i, 1/n)\Big)$ are $1/n$ $d$-apart, then  $$\xi\in {\rm cl}_{s_d X} \Big(\bigcup_{i \in F\cap F\,'}B^{\infty}_{d}(x_i, 1/n)\Big).$$

Finally, $\mathfrak{F}$ is a filter in $I_{1/n}$ since clearly  when $F\subset F\,'\subset I_{1/n}$ and $F\in \mathfrak{F}$, then $F\,'\in \mathfrak{F}$. That   $\mathfrak{F}$ is an ultrafilter follows immediately  since for every $F\in I_{1/n}$, we have that $$X= \Big(\bigcup_{i \in F}B^{\infty}_{d}(x_i,1/n)\Big)\cup   \Big(\bigcup_{i \in I_{1/n}\setminus F}B^{\infty}_{d}(x_i,1/n)\Big).$$ Then   $X$ is the disjoint union of two sets  $1/n$ $d$-apart  and therefore $F$ or $I_{1/n}\setminus F$ is in $\mathfrak{F}$.

\smallskip

In order to see that  $\mathfrak{F}$ is $\alpha$-bounded, let $f\in U_{d}(I_{1/n},\Nset)= U_d^n(X,\Nset)$, and take the metric $\rho_{n,f}$. According to Theorem \ref{Samuel as supremum Lipschitz-2}, we have that $\xi \in H(Lip_{\rho_{n,f}}(X))$, and then $\xi\in {\rm cl}_{s_d X} B$ for some $\rho_{n,f}$-bounded subset $B$ (Proposition \ref{characterization-H(Lip_d(X))}). Applying now Lemma \ref{bounded n,f 1}, there exist $F\subset I_{1/n}$ with  $f(\{x_i :i\in F\})$  finite, and  $M\in \Nset$ satisfying that $B\subset \bigcup _{i \in F} B^{M}_{d}(x_ i , 1/n).$
That is, $f$ is bounded on some member $F$ of $\mathfrak{F}$. Therefore $\mathfrak{F}$ is an $\alpha$-bounded ultrafilter.

Now, from Theorem \ref{Samuel-realcompactness for unif. discrete}, $\mathfrak{F}$ must be fixed, and then there exists (a unique) $i_0\in I_{1/n}$ in the intersection of all the sets in $\mathfrak{F}$. And we finish since $\xi\in {\rm cl}_{s_d X} B^{\infty}_{d}(x_{i_0}, 1/n)$, as we wanted.
\end{proof}
\end{lemma}

\smallskip

\begin{theorem} \label{characterization Samuel realcompactness} A metric space $(X,d)$ is Samuel realcompact if and only if it is Bourbaki-complete and every uniformly discrete subspace of  $X$ has non-measurable cardinal.
\begin{proof} Suppose $(X,d)$ is Samuel realcompact. Hence, it is realcompact, and then every discrete closed subspace has non-measurable cardinal (see \cite{gillman}). In particular, since  every uniform  discrete subspace  is closed it must have non-measurable cardinal. Now, in  order to analyze the Bourbaki-completeness of $X$ we are going to apply last Proposition \ref{Bourbaki-complete and Bourbabi-bounded}. So, let $B$ any closed  Bourbaki-bounded subset of $X$. Then, from Proposition \ref{characterization-bourbaki-bounded}, $B$ is $\rho$-bounded for every metric $\rho$ uniformly equivalent to $d$. Hence $${\rm cl}_{s_d X}B\subset \bigcap_{\rho\stackrel{{u}}{\sim} d} H(Lip_\rho(X))=H(U_d (X))=X.$$ That is, every closed and Bourbaki-bounded subset of $X$ is compact, as we wanted.

\smallskip

Conversely,  let  $\xi \in H(U_d(X))$,  and let  $(x_\lambda)_{\lambda\in \Lambda}$ be a net in $X$ converging to $\xi$. Clearly, since $X$ is Bourbaki-complete, we finish if we see that this net is Bourbaki-Cauchy. So, let $\varepsilon>0$ and take  $n\in \Nset$ with $1/n<\varepsilon$. Apply  Lemma \ref{Non-measurable cardinal I_1/n}, since from the hypothesis the uniformly discrete subspace $I_{1/n}$ must have non-measurable cardinal,  let $B^\infty_{d}(x_{i_0},1/n)$ the unique $1/n$-chainable component containing $\xi$ in its closure.

Now, consider the metric $\rho_{n,f}$ when $f\equiv 1$ is the constant function.  We know that $\xi\in H(Lip_{\rho_{n,f}}(X))$ (from Theorem \ref{Samuel as supremum Lipschitz-2}), and hence $\xi\in {\rm cl}_{s_dX}B$ for some $\rho_{n,f}$-bounded subset $B$ (Proposition \ref{characterization-H(Lip_d(X))}). Now Lemma \ref{bounded n,f 1} ensures the existence of   $F\subset I_{1/n}$ and  $M\in \Nset$ satisfying that $B\subset \bigcup _{i \in F} B^{M}_{d}(x_ i , 1/n).$
Since $B\cap B^\infty_{d}(x_{i_0},1/n)$ and $B\setminus B^\infty_{d}(x_{i_0},1/n)$ are 1/n $d$-apart, then we deduce that $$\xi\in {\rm cl}_{s_d X}\big(B\cap B^\infty_{d}(x_{i_0},1/n)\big)\subset {\rm cl}_{s_d X}\big ( B^{M}_{d}(x_ {i_0}, 1/n)\big ). $$

We assert that there exists  $\lambda_0\in \Lambda$, such that for $\lambda\geq \lambda_0$,  $x_\lambda\in B^{M+1}_{d}(x_ {i_0}, 1/n)$. Otherwise,  $\xi$ would be in the closure of two sets $1/n$ $d$-apart, namely $ B^{M}_{d}(x_ {i_0}, 1/n)$ and $X\setminus  B^{M+1}_{d}(x_ {i_0}, 1/n))$, which is impossible. And we finish since $(x_\lambda)_\lambda$ is Bourbaki-Cauchy, as we wanted.
\end{proof}
\end{theorem}

\begin{remark} \label{remark-->uniform partition} Observe that in the above  proof we only use the non-measurable cardinality of the sets $I_{1/n}$, for every $n$, which is in fact equivalent to the property  that ``every uniform partition of $X$ has non-measurable cardinality''. Clearly, spaces having  this property are for instance every connected metric space, or more generally every uniformly connected, and also every separable metric space.
\end{remark}

Therefore, according with this remark, last theorem can be rewritten as follows.

\begin{theorem} \label{characterization Samuel realcompactness 2} A metric space $(X,d)$ is Samuel realcompact if and only if it is Bourbaki-complete and every uniform partition of  $X$ has non-measurable cardinal.
\end{theorem}

Nevertheless, some condition of non-measurable cardinality is needed in Theorems \ref{characterization Samuel realcompactness} and \ref{characterization Samuel realcompactness 2}.  Indeed, if $X$ is a set with a measurable cardinal endowed with the $0-1$ metric then it is Bourbaki-complete but not Samuel realcompact since in fact it is not realcompact. As it is well known, in the absence of measurable cardinals, every metric space is realcompact (see \cite{gillman}). But the same is not true for  Samuel realcompactness. For instance, if the metric space is not complete, then it can not be Bourbaki-complete nor Samuel realcompact.

\smallskip

In fact,  realcompactness  and Samuel realcompactness are properties that can be very far away, as the next easy result shows.

\begin{corollary} A Banach space is Samuel realcompact if and only if it has finite dimension.
\end{corollary}

\begin{proof} Firstly, it is clear that  every finite dimensional Banach space is Samuel realcompact since in fact it is Lipschitz realcompact. Conversely, if a space is Samuel realcompact  then it is Bourbaki complete (Theorem \ref{characterization Samuel realcompactness}), but as we have said before a normed space is Bourbaki-complete if and only if it has finite dimension.
\end{proof}

\smallskip

It is interesting to say here that different authors have obtained some kind of uniform Katetov-Shirota results. More precisely, in \cite{isbell}, Isbell proved that, for the particular case of  locally fine uniform spaces (see the definition in \cite{isbell}) without uniformly discrete subspaces of measurable cardinal, completeness implies the completeness of the $c$-modification of $X$. Later, Rice in \cite{rice}, and Reynolds and Rice in \cite{reynolds} demonstrated the analogous result but for the particular cases of uniform spaces satisfying that the  family of real-valued uniformly continuous functions is  closed under inversion, and also for  uniform spaces having a star-finite basis, always assuming that each closed (uniformly) discrete subspace of them has non-measurable cardinality. In the same line,  Hu\v{s}ek and Pulgar\'{\i}n in \cite{husek}, proved the same for uniformly 0-dimensional spaces without uniformly discrete subsets of measurable cardinal, where a uniform space is {\it uniformly 0-dimensional} if it has a base for the uniformity composed of partitions (for instance, every uniformly discrete space is uniformly 0-dimensional).
Finally,  N{\aa}jstad gave another characterization of  realcomplete proximity spaces (\cite{najstad}) but in a very  different style from  ours.
\smallskip

Leaving behind the discussion about the measurability or non-measurability of cardinals, we can say that Theorem \ref{characterization Samuel realcompactness} establishes the equivalence between two uniform properties in the frame of metric spaces, namely Samuel realcompactness and Bourbaki-completeness. In such a way that the study made in \cite{merono2} for Bourbaki-completeness can be used here in order to get more information  about Samuel realcompactness. For instance, we know that this property is hereditary for closed subspaces, and  also countably productive. Thus, spaces like $\Nset^\Nset$
and $\Rset^\Nset$ endowed with the corresponding product metrics, are Samuel realcompact. Moreover the problem of when there is, for a metrizable space, some  compatible  metric making it Samuel realcompact, will be now  equivalent to know when the space is what it is called  Bourbaki-completely metrizable. And therefore, from \cite{merono2},  we can obtain an answer to this question in the next result  that is in the line of the well know \v{C}ech theorem  saying that a metrizable space $X$ is completely metrizable if and if it is a $G_\delta$-set in $\beta X$ (see \cite{engelkingbook}).

\smallskip

\begin{theorem} Let $(X,\tau)$ be a metrizable space (with a non-measurable cardinal). Then there exists a compatible metric $d$ such that $(X,d)$ is Samuel realcompact if and only if $X = \bigcap_{n=1}^\infty G_n$ where each $G_n$ is an open and paracompact subspace of $s_dX$.
\end{theorem}

\section{Samuel realcompactification and completion}

 \smallskip

At this point of the paper,  it  seems natural to analyze the relationship between the Samuel realcompactification of a metric space $(X,d)$ and the Samuel realcompactification of its completion $(\widetilde X,\widetilde d)$. First of all recall that, as was proved by Woods in \cite{woods},  the analogous question for the compactifications has an elegant answer, namely   $s_d X$ and $s_{\widetilde d} \, \widetilde X$ are equivalent compactifications of $X$. Observe, at this respect,  the difference between this compactification and the Stone-\v{C}ech compactification.

\smallskip

Thus, if we identify  $s_d X$ and $s_{\widetilde d} \, \widetilde X$ and we write $X\subset \widetilde X\subset H(U_{\widetilde d}\, (\widetilde X))\subset s_d X$, then the following result  can be easily derived. Note that this result given in terms of the corresponding  $c$-modifications and completions  is essentially contained in \cite{ginsburg}.

\begin{proposition} \label{equal Samuel realc.} Let  $(X,d)$ be a metric space and $(\widetilde X,\widetilde d)$ its completion. Then $H(U_{d}(X))$ and  $H(U_{\widetilde d}\, (\widetilde X))$ are equivalent realcompactifications of $X$.
\end{proposition}
\begin{proof} The proof follows using  that  $U_d(X)=U_{\widetilde d}\, (\widetilde X)|_X$, and the equivalence $s_d X\equiv s_{\widetilde d} \, \widetilde X$. Indeed, $$H(U_{\widetilde d}\, (\widetilde X))=\{\xi\in s_{\widetilde d} \, \widetilde X:  \, f^*(\xi)\neq \infty \text{\,\, for all \,} f\in U_{\widetilde d}\, (\widetilde X)\big\}\equiv$$ $$ \hspace{2.1cm} \equiv \{\xi\in s_d X:  \, f^*(\xi)\neq \infty \text{\,\, for all \,} f\in U_{\widetilde d}\, (\widetilde X)|_X\big\}=$$ $$\hspace{3.5cm} =\{\xi\in s_d X:  \, f^*(\xi)\neq \infty \text{\,\, for all \,} f\in U_d(X)\big\}=H(U_d(X)).$$
\end{proof}

\begin{remark} \label{equal Lipschitz realc.} Since  $Lip_d(X)=Lip_{\widetilde d}\, (\widetilde X)|_X$, with an analogous proof to the above we  can also  derive that $H(Lip_{d}(X))$ and  $H(Lip_{\widetilde d}\, (\widetilde X))$ are equivalent realcompactifications of $X$.
\end{remark}

Now, we are interested in knowing  when the Samuel realcompactification of a metric space $(X,d)$ is just its completion $\widetilde X$. Recall that, as we have already mentioned in Example \ref{example 2},   Woods proved in \cite {woods} that   $s_d X=\widetilde X$ if and only if $X$ is a totally bounded metric space, or equivalently $\widetilde X$ is compact. We will see that for $H(U_d(X))$ the condition that appears will be the  Bourbaki-completeness of $\widetilde X$. But, firstly we need the following easy lemma.

\begin{lemma} \label{Bourbaki-bounded sets are equal} Let  $(X,d)$ be a metric space and $(\widetilde X,\widetilde d)$ its completion. A subset  $B\subset X$  is Bourbaki-bounded in $X$ if and only if it  is Bourbaki-bounded in $\widetilde X$.
\end{lemma}

\begin{proof} The proof follows at once using Proposition \ref{Bourbaki-bounded versus uniform functions}, and again that $U_d(X)=U_{\widetilde d}\, (\widetilde X)|_X$.
\end{proof}

\begin{proposition} \label{Samuel versus completion} Let  $(X,d)$ be a metric space and $(\widetilde X,\widetilde d)$ its completion. The following are equivalent:
\begin{enumerate}
\item $H(U_d(X))=\widetilde X $.
\item $(\widetilde X, \widetilde d)$ is Bourbaki-complete and  every uniformly discrete subspace of $X$ has non-measurable cardinal.
\item Every Bourbaki-bounded subset of $X$ is totally bounded and  every uniformly discrete subspace of $X$ has non-measurable cardinal.
\end{enumerate}
\end{proposition}

\begin{proof}  The equivalence $(1) \Leftrightarrow (2)$ follows by using properly Proposition \ref{equal Samuel realc.},   Theorem   \ref{characterization Samuel realcompactness},  and taking into account that  every uniformly discrete subspace of $X$ has non-measurable cardinal iff  the same is true for $\widetilde X$.

$(2) \Rightarrow (3)$ If $B$ is a Bourbaki-bounded subset of $X$, it is easy to see that ${\rm cl}_{\widetilde X} B$ is  Bourbaki-bounded in $\widetilde X$. Now, from Proposition \ref{Bourbaki-complete and Bourbabi-bounded}, ${\rm cl}_{\widetilde X} B$ is  compact, and then $B$ is totally bounded.

$(3) \Rightarrow (2)$ To see that $\widetilde X$ is Bourbaki-complete we will use the (equivalent) definition of this property  given by sequences. So, let  $(y_n)_n$ be a Bourbaki-Cauchy sequence in $\widetilde X$. For each $n\in \mathbb N$, choose $x_n\in X$ with $\widetilde d(x_n, y_n)<1/n$. It is not difficult to check  that  $B=\{x_n: n\in \mathbb N\}$ is in fact a Bourbaki-bounded subset in $\widetilde X$. Then, from  Lemma \ref{Bourbaki-bounded sets are equal}, $B$ is Bourbaki-bounded in $X$, and by condition $(3)$ $B$ is totally bounded. By completeness, we deduce that  ${\rm cl}_{\widetilde X} B$ is compact and therefore the sequence $(x_n)_n$ clusters in $\widetilde X$. Clearly, the same happens with the sequence $(y_n)_n$, as we wanted.
\end{proof}

\smallskip

Next, we are going to  characterize the metrizability of the Samuel realcompactification. For that we will use another result by Woods asserting  that   $s_d X$ is metrizable if and only if $X$ is a totally bounded metric space (\cite{woods}). Moreover  we will need  the following lemma that can be also seen in \cite{woods}.

\begin{lemma} {\rm ({\sc Woods} \cite{woods})} \label{Samuel compactification of any B} Let  $(X,d)$ be a metric space and $B\subset X$. Then $s_d B$, the Samuel compactification of $B$,  and ${\rm cl}_{s_d X}B$ are equivalent compactifications of $B$.
\end{lemma}

\begin{theorem} \label{metrizability of Samuel} Let $(X,d)$ be a metric space and  $(\widetilde X,\widetilde d)$ its completion. Then $H(U_{d}(X))$ is metrizable if and only if
$H(U_d(X))= \widetilde X$.
\end{theorem}
\begin{proof} Clearly if $H(U_d(X))= \widetilde X$, the Samuel realcompactification of $X$ is metrizable.

Conversely, suppose  $H(U_d(X))$ is metrizable. Since  any realcompact space where every point is $G_\delta$ is hereditary realcompact (see \cite{gillman}), it follows that   $\widetilde X\subset H(U_d(X))$ is realcompact.  Therefore any uniformly discrete subspace of $\widetilde X$ has non-measurable cardinal. Now, if  $H(U_d(X))\neq \widetilde X$, then from above Proposition \ref{Samuel versus completion}, there exists some $B\subset X$ which is Bourbaki-bounded in $X$ but not totally bounded. According to the above mentioned result by Woods, we know that the Samuel compactification  of $B$, i.e. $s_d B$ is not metrizable. We are going to see that in fact $s_d B$ is a subspace of  $H(U_d(X))$, and therefore  $H(U_d(X))$ can not be metrizable.

Now, according to the equivalence $s_d B\equiv{\rm cl}_{s_d X}B$ (Lemma \ref{Samuel compactification of any B}), we finish if we prove  that ${\rm cl}_{s_d X}B$ is in fact contained in  $H(U_d(X))$. For that, take $\xi\in {\rm cl}_{s_d X}B$. To see that $\xi  \in H(U_d(X))$ it is enough to make sure that $f^*(\xi)\neq \infty$, for every $f\in U_d(X)$. But this is clear since every uniformly continuous function $f$ must be bounded on the Bourbaki-bounded set $B$ (Proposition \ref{Bourbaki-bounded versus uniform functions}).
\end{proof}

We finish this section with an analogous result to the above for  $H(Lip_d(X))$.

\begin{theorem} Let $(X,d)$ be a metric space and  $(\widetilde X,\widetilde d)$  its completion. Then $H(Lip_{d}(X))$ is metrizable if and only if
$H(Lip_d(X))= \widetilde X$.
\end{theorem}
\begin{proof} One implication is clear. To the converse, suppose $H(Lip_d(X))$ is metrizable, then for every $x\in  X$ and $\varepsilon > 0$  we have that ${\rm cl}_{s_d X} (B_d[x, \varepsilon])\equiv s_d B_d[x, \varepsilon]$ (Lemma \ref{Samuel compactification of any B}) is metrizable. Hence again by Woods, we follow that every closed  ball in $X$ is totally bounded and also that $s_d B[x, \varepsilon]$ is just its completion. Since the completion of every set in $X$ is clearly contained in $\widetilde X$, then  we have that $\widetilde X\subset H(Lip_d(X)) = \bigcup_{n\in \mathbb N} {\rm cl}_{s_d X} (B[x, n])  \subset  \widetilde X$, as we wanted.
\end{proof}

\smallskip

\section{Some results related to bornologies}

\smallskip

In this section we are going to summarize, in a synoptic table, some results about realcompactifications in terms of some kind of bornologies that we can consider in a metric space. Recall that a family $\bf B$ of subsets of a non-empty set $X$ is said to be  a {\it bornology} in $X$ when it satisfies the following conditions:
(i) For every $x\in X$, the set  $\{x\}\in \rm\bf B$; (ii) If $B\in \bf B$ and $A\subset B$, then $A\in\bf B$ and (iii) If $A, B\in\bf B$, then $A\cup B\in \bf B$.  Moreover, if $X$ is a topological space, we say that ${\bf B}$ is a  {\it closed} bornology when, (iv) If $B\in \bf B$ then its closure $\overline{B}\in \bf B$.

\smallskip

Thus, if we denote by  $\textbf{B}_{d}(X)$ the $d$-bounded subsets,  $\textbf{TB}_{d}(X)$ the totally-bounded subsets, $\textbf{BB}_{d}(X)$ the Bourbaki-bounded subsets,  $\textbf{CB}_d(X)$ the so-called compact  bornology, i.e., the subsets in $X$ with compact closure, and finally by $\textbf{P}(X)$ the usual power set of $X$,  then it is easy to check that all these families are closed bornologies in the metric space $(X,d)$. And, clearly, we have that, $$\textbf{CB}_d(X) \subset \textbf{TB}_d(X)\subset \textbf{BB}_d(X)\subset \textbf{B}_d(X) \subset \textbf{P} (X).$$

In general, these families are different from each other. For instance, if we take $X=\mathbb Q \times \ell _2$, endowed with the metric $d=\sup \{d^*, \|\cdot\|\}$  where $d^*=\inf\{1, d_{u}\}$, then all of these bornologies on  this metric space are different (see \cite{merono1}).

Our objective here is to show  that an equality between the above  bornologies, provides an  equality between the following topological spaces, $$X\subset \widetilde X\subset H(U_d (X))\subset H(Lip_d (X)) \subset s_d X.$$

In the next double entry table we are going to collect the results in the line mentioned above. Note that this table must be read as follows. The numbers {\bf(1)}, {\bf(2)}, ...,  denote each of the ordered (by inclusion)  bornologies $\textbf{CB}_d(X),  \textbf{TB}_d(X), ...$. Thus, an space in the first column is equal to an space in the first
row if and only if it is true the equality of the bornologies appearing in the box formed by the intersection
of the corresponding row and column where they are. For instance,  $X=\widetilde X$ if and only if ${\bf (1)=(2)}$, i.e, every totally bounded subset if $X$ has compact closure. Note that for the (two) results in the column under the space
$H(U_d(X))$ we need to suppose some  additional condition of non-measurable cardinality since we are applying Theorem \ref{characterization Samuel realcompactness}. Namely, we  will denote  by  {\small $\spadesuit$}, the condition on $X$ that every uniformly discrete subspace has non-measurable cardinal or equivalently in this setting every uniform partition of $X$ has non-measurable cardinal (see Remark \ref{remark-->uniform partition}).
\bigskip

{\centerline{
\begin{tabular} { |c||c|c|c|c| }
\hline
&  &  &   &   \\
& $\widetilde X$  & $H(U_d(X))$ & $H(Lip_d(X))$  & $s_d X$ \\
\hline
\hline
&  & \hspace{3cm}{\small $\spadesuit$} &   &   \\
$X$ & \,\,\,\, \bf {(1)=(2)} \,\,\,\,  & \bf{(1)=(3)} & \bf{(1)=(4)}  & \bf{(1)=(5)} \\
& \,\,\,\,\,  {\small $X$ complete} \,\,\,\,\, & {\small  $X$ Bourbaki-complete} &  \,\,\,\,\, {\small  $X$  Heine-Borel} \,\,\,\,\,\,  & {\small $X$  compact}   \\
\hline
&  & \hspace{3cm}{\small $\spadesuit$} &   &   \\
$\widetilde X$ &  & \bf {(2)=(3)} & \bf {(2)=(4)}  & \bf {(2)=(5)} \\
& & {\small $\widetilde X$ Bourbaki-complete} & {\small $\widetilde X$ Heine-Borel}  & {\small $X$  totally-bounded}   \\
\hline
&  &  &   &   \\
$H(U_d(X))$ &  &  & \bf {(3)=(4)}  & \bf {(3)=(5)} \\
& &  &   & {\small $X$  Bourbaki-bounded}   \\
\hline
&  &  &   &   \\
$H(Lip_d(X))$ &  &  &   & \bf {(4)=(5)} \\
& &  &   & {\small $X$ bounded}   \\
\hline
\end{tabular}}}

\bigskip

\medskip

We will prove all the results contained in this table, assuming {\small $\spadesuit$} when we apply Theorem \ref{characterization Samuel realcompactness}.

\smallskip

\begin {enumerate}
\item [$\bullet$] $X=\widetilde X$ $\Longleftrightarrow$  \fbox{$X$ is complete} $\Longleftrightarrow$   $\textbf{CB}_d(X) =  \textbf{TB}_d(X)$.

\smallskip

\item [$\bullet$] $X= H(U_d(X))$ $\Longleftrightarrow$  $X$ is Samuel realcompact  $\stackrel{(\text{Th. }
\ref{characterization Samuel realcompactness})}{\Longleftrightarrow}$  \fbox{$X$ is Bourbaki-complete} $\stackrel{(\text{Prop. }
\ref{Bourbaki-complete and Bourbabi-bounded})}{\Longleftrightarrow}$ $\textbf{CB}_d(X) =  \textbf{BB}_d(X)$.

\smallskip

\item [$\bullet$]  $X= H(Lip_d(X))$ $\Longleftrightarrow$  $X$ is Lipschitz realcompact  $\stackrel{(\text{Prop. } \ref{Lipschitz-realcompact})}{\Longleftrightarrow}$
    \fbox{$X$ is Heine-Borel} $\Longleftrightarrow$ $\textbf{CB}_d(X) =  \textbf{B}_d(X)$.

\medskip

\item [$\bullet$] $X=s_d X$ $\Longleftrightarrow$  \fbox{$X$ is compact}  $\Longleftrightarrow$ $\textbf{CB}_d(X) =  \textbf{P}(X)$.

\smallskip

\item [$\bullet$] $\widetilde X=H(U_d(X))$ $\stackrel{(\text{Th. } \ref{equal Samuel realc.})}{\Longleftrightarrow}$ $\widetilde X=H(U_{\widetilde d}(\widetilde X))$ $\Longleftrightarrow$ $\widetilde X$ is Samuel realcompact  $\stackrel{(\text{Th. }
    \ref{characterization Samuel realcompactness})}{\Longleftrightarrow}$

    \fbox{$\widetilde X$ is Bourbaki-complete} $\stackrel{(\text{Prop. }
    \ref{Samuel versus completion})} {\Longleftrightarrow}$  $\textbf{TB}_d(X) = \textbf{BB}_d(X)$.

\smallskip

\item [$\bullet$]  $\widetilde X= H(Lip_d(X))$  $\stackrel{(\text{Rem. } \ref{equal Lipschitz realc.})}{\Longleftrightarrow}$ $\widetilde X=H(Lip_{\widetilde d}(\widetilde X))$ $\Longleftrightarrow$ $\widetilde X$ is Lipschitz realcompact  $\stackrel{(\text{Prop. } \ref{Lipschitz-realcompact})}{\Longleftrightarrow}$

    \fbox{$\widetilde X$ is Heine-Borel} $\Longleftrightarrow$ $\textbf{TB}_d(X) =  \textbf{B}_d(X)$.

\smallskip

\item [$\bullet$] $\widetilde X=s_d X$ $\Longleftrightarrow$  $\widetilde X=s_{\widetilde d} \widetilde X$ $\Longleftrightarrow$ $\widetilde X$ is compact $\Longleftrightarrow$ \fbox{$X$ is totally bounded}  $\Longleftrightarrow$ $\textbf{TB}_d(X) =  \textbf{P}(X)$.

\smallskip

\item [$\bullet$]  $H(U_d(X))= H(Lip_d(X))$  $\stackrel{(\text{Prop. } \ref{Samuel real. = Lipschitz real.})}{\Longleftrightarrow}$
$\textbf{BB}_d(X) =  \textbf{B}_d(X)$.

\smallskip

\item [$\bullet$]  $H(U_d(X))= s_d X$  $\stackrel{(\text{Prop. } \ref{Samuel realcompactification = Samuel compactification})}{\Longleftrightarrow}$\fbox{$X$ is Bourbaki-bounded}  $\Longleftrightarrow$
$\textbf{BB}_d(X) =  \textbf{P}(X)$.

\smallskip

\item [$\bullet$]  $H(Lip_d(X))= s_d X$  $\stackrel{(\text{Prop. }
\ref{Lipschitz realcompact. = Samuel compact.})}{\Longleftrightarrow}$ \fbox{$X$ is bounded}  $\Longleftrightarrow$
$\textbf{B}_d(X) =  \textbf{P}(X)$.

\end {enumerate}

\medskip

Note that in each box of the  table,  it also appears some internal property of the metric space $X$ or of its completion $\widetilde X$ characterizing the corresponding situation,  except for the case $H(U_d(X))= H(Lip_d(X))$. In fact, we wonder if there exists some known property defined for metric spaces that determines this equality, or equivalently, the equality  $\textbf{BB}_d(X) =  \textbf{B}_d(X)$.  Examples of such spaces are of course any metric space with the Heine-Borel property, and also the so-called
small-determined metric spaces introduced in [7]. As we have already mentioned in Section 1, these small-determined spaces, that includes all the normed spaces and  all the length spaces, are characterized by the fact that every real-valued uniformly continuous function can be uniformly approximated by Lipschitz functions. Then clearly for these spaces the Samuel and the Lispchitz realcompactifications coincide, but  the converse is not true. For instance, if we take the set of natural numbers $\Nset$ endowed with the usual metric, then the space $\Nset \times \ell _2$ satisfies that every bounded set is Bourbaki-bounded, but it is not small-determined neither it satisfies the Heine-Borel property. We refer to the paper \cite{merono3} where we see that spaces for which $\textbf{BB}_d(X) =  \textbf{B}_d(X)$ have the property that every uniform partition is in fact countable, and also that they are properly located between  two known classes of metric spaces, namely the above mentioned small-determined spaces  and the so-called $B$-simple spaces introduced by Hecjman in \cite{hejcman2}.

\smallskip

\section{Samuel and  Hewitt-Nachbin realcompactifications}

\smallskip

As we have seen along the paper  $\upsilon X$ and $H(U_d(X))$  can be different realcompactifications for the metric space $X$. Then, a natural question is to know when they are equivalent realcompactifications. We know that $\upsilon X\geq H(U_d (X))$ considering the usual order in the family of realcompactifications. In principle, these realcompactifications live in different compactifications, namely $\beta X$ and $s_d X$, respectively, then the comparison between them may not be equivalent to the corresponding inclusion. Nevertheless, we are going to see that in fact both live in  $s_d X$, and then we will derive that  $\upsilon X\subset H(U_d(X))$.

\smallskip

First of all recall that, for every Tychonoff space $X$,   $\upsilon X$ is the $G_\delta$-closure of $X$ in $\beta X$ (see for instance \cite{gillman}). On the other hand, we say that the  topological space $X$ is  {\it $z$-embedded} in the space $Y$, whenever $X\subset Y$ and each zero-set of $X$ is the restriction to $X$ of a zero-set in $Y$. For instance, any metric space $(X,d)$ is clearly  $z$-embedded  in $s_d X$ and therefore in any $Y$ with $X\subset Y\subset s_d X$. Moreover, in connection with this notion, we will use a result by Blair and Hager    saying that under $G_\delta$-density assumption, $z$-embedding and $C$-embedding are equivalent properties (see \cite{blair-hager}). Then with all of these ingredients we have the following.

\begin{proposition} \label{upsilon subset Samuel} Let $(X,d)$ be a metric space. Then we have,
\begin{enumerate}

\item [(i)] $\upsilon X$ is a realcompactification of $X$ equivalent to the  $G_\delta$-closure of $X$ in $s_d X$.

\item [(ii)] $\upsilon  X\subset H(U_d(X))$.

\item [(iii)] $X$ is realcompact if and only if $X$ is $G_\delta$-closed in $s_d X$.

\end{enumerate}

\end{proposition}

\begin{proof} Let $Y$ be the $G_\delta$-closure of $X$ in  $s_d X$. Then $Y$ is a realcompact space since it is known that every $G_\delta$-closed subspace of a (real)compact space is also realcompact (see \cite{gillman}). In particular, we have that $Y$ is a realcompactification of $X$ where it is $G_\delta$-dense. As we have said before, $X$ is clearly $z$-embedded in $s_d X$ and also in $Y$. Now, by using the above mentioned result by Blair and Hager (\cite{blair-hager}), we have that $X$ must be $C$-embedded in $Y$.  And we finish, since $\upsilon  X$ is the unique (up equivalence) realcompactification of $X$ where it is $C$-embedded (see \cite{engelkingbook}). Therefore $Y$ is equivalent to $\upsilon X$, and then  (i) follows.

Now, in order to see property  (ii), we identify $\upsilon X$ with the equivalent copy of it contained in  $s_d X$. Thus, we have that  both realcompactifications  $\upsilon X$ and $H(U_d(X))$ are contained in $s_d X$, and then the relation $\upsilon X\geq H(U_d (X))$ can be write as $\upsilon X\subset H(U_d (X))$.

And, finally (iii) is immediate from (i).
\end{proof}

\begin{remark}

In \cite{curzer} Curzer and Hager introduced the class $\mathcal{K} _1$ of those uniform spaces $(X,\mu)$ which are $G_\delta$-closed in their Samuel compactification $s_{\mu} X$.  Note that the same class of uniform realcompact spaces has been considered by Chekeev in \cite{chekeev}. Then from the above result (iii), we can say that   a metrizable space is a member of this class $\mathcal{K} _1$ if and only if it is realcompact. Therefore, it is clear that the uniform concept of realcompactness managed by Curzer, Hager and Chekeev is far from Samuel realcompactness. In fact, it may  lack Bourbaki-completeness when it is a  metric space.

\end{remark}

Going back to our initial problem, it is very clear that $\upsilon X$ and $H(U_d(X))$ will coincide  whenever every real-valued continuous function is uniformly continuous, i.e., when the space is said to be a {\it UC}-space  or also called {\it Atsuji} space. We refer to  the nice paper by Jain and Kundu  \cite{kundu}, where  many different characterizations of these spaces are given. As we can see in the next result,  it is precisely in the frame of the $UC$-spaces where  the corresponding equivalence between $\beta X$ and $s_d X$ occurs.

\begin{theorem} {\rm ({\sc Woods}  \cite{woods})} For a metric space $(X,d)$, the following statements are equivalent:
\begin{enumerate}
\item $X$ is a $UC$-space.

\item $C(X)=U_{d}X$.

\item $C^{*}(X)=U^{*}_{d}X$.

\item $\beta X\equiv s_{d}X$.

\end{enumerate}
\end{theorem}

According to the last theorem, one can expect that also the equivalence between $\upsilon X$ and $H(U_d(X))$  will only  occur within the $UC$-spaces. But it is not true. For that it is enough to consider  $\Rset$ with the usual metric, since it is Samuel realcompact (and realcompact), i.e.,  $\Rset =\upsilon \Rset=H(U_d(\Rset))$, but it is not a $UC$-space. Moreover, next example makes clear that   even for discrete metric spaces there is not an analogous result to the above.

\begin{example} Let $X=\{1, 1+1/2, 2, 2+1/3, 3, 3+1/4, ...\}$ be endowed with the usual metric. Clearly  $X$ is a discrete metric space but not $UC$. Note that a discrete metric space is $UC$ if and only if it is uniformly discrete. Now, since $X$ is Heine-Borel then it is in fact Lipschitz realcompact, and so  $X=\upsilon X=H(U_d(X))$.
\end{example}

Then a natural  question here is whether for a metrizable space there exists some compatible (topological equivalent) metric making it a $UC$-space. The answer was given by Beer in \cite{beer}, where he proved that a metrizable space $X$ admits a compatible $UC$ metric if and only if the set $X^\prime$ of the accumulation points of $X$ is compact. That means in particular that every discrete space $X$ admits a  $UC$ metric $d$, and then $\upsilon X= H(U_d (X))$. However  the real line $\mathbb R$ with the usual topology  admits no $UC$ metric.

\medskip

Now we are going to state our main result in this section, asserting that the equivalence between $\upsilon X$ and $H(U_d(X))$ only occurs whenever $X$ is somehow well-placed in $H(U_d(X))$, namely when it is $G_\delta$-dense in its Samuel realcompactification. Recall that for every unital vector lattice $\mathcal L$ of continuous functions on $X$,  it is true that $H(\mathcal L)$ is $G_\delta$-closed in $H(\mathcal L^*)$ and then the $G_\delta$-closure of $X$ in  $H(\mathcal L^*)$ is contained  in $H(\mathcal L)$, but the space $X$ is not necessarily $G_\delta$-dense in $H(\mathcal L)$  (see \cite{garrido1}). In our case, that is for $\mathcal L=U_d(X)$, we can say  that  $H(U_d(X))$ is $G_\delta$-closed in $s_d X$ and therefore the  $G_\delta$-closure of $X$ in $s_d X$ is contained in $H(U_d(X))$. But in addition we have the following.

\begin{theorem} \label{upsilon = Samuel}  For a metric space $(X,d)$ the following statements are equivalent:
\begin{enumerate}

\item $\upsilon  X= H(U_d(X))$.

\item $X$ is $G_\delta$-dense in $H(U_d(X))$.

\item $H(U_d(X))$ is the $G_\delta$-closure of $X$ in $s_d X$.

\end{enumerate}
\end{theorem}
\begin{proof} (1) implies (2) follows at once since $X$ is always $G_\delta$-dense in $\upsilon X$. For (2) implies (3), it is enough to note that $H(U_d(X))$ is $G_\delta$-closed in $s_d X$, as we have mentioned in the above paragraph.  Finally that (3) implies (1) follows from condition (i) in Proposition \ref{upsilon subset Samuel}.
\end{proof}

Note that with an analogous proof to the above we
can also  derive an analogous result for $H(Lip_d(X))$.

\begin{theorem} \label{upsilon = Lipschitz realc.} For a metric space $(X,d)$ the following statements are equivalent:
\begin{enumerate}

\item $\upsilon  X= H(Lip_d(X))$.

\item $X$ is $G_\delta$-dense in $H(Lip_d(X))$.

\item $H(Lip_d(X))$ is the $G_\delta$-closure of $X$ in $s_d X$.

\end{enumerate}

\smallskip

\end{theorem}

\begin{remark} Note that the equality $\upsilon  X = H(U_d(X))$ implies the completeness of $X$. Indeed, any point in $\widetilde X\setminus X$ is in  $H(U_d(X))$ but not in the $G_\delta$-closure of $X$ in $s_d X$, since it is itself a $G_\delta$-set in $\widetilde X$ and then it must be contained in some $G_\delta$-set in $s_d X$ which does not meet $X$. On the other hand, the completeness of $X$ does not imply the equality $\upsilon  X=H(U_d(X))$, for that take for instance the Banach space $X=\ell_2$.

\end{remark}

Next example shows how can be different all the spaces associated to a metric space $(X,d)$ appearing in this paper, namely $X, \widetilde X, \upsilon X, H(U_d(X))$, $H(Lip_d(X))$, $s_d X$ and $\beta X$. Note that in order to have $X\neq \upsilon X$ we need to assume the  existence of measurable cardinals.

\begin{example} Let   $Y$ a set with measurable cardinal with the $0-1$ metric. Consider  the product metric space $X=Y\times (\ell_2\setminus\{0\})$. Clearly,  $X$ is a not realcompact neither complete, and therefore  we have $X \varsubsetneq  \upsilon X \varsubsetneq H(U_d(X))$ (see last Remark). Moreover, $H(U_d(X))\varsubsetneq H(Lip_d(X))\varsubsetneq  s_d X$, since $B=Y\times\{x\}$, for $0\neq x\in \ell_2$, is a bounded subset of  $X$ which is not  Bourbaki-bounded and also  $X$ is not bounded (see the table in Section 7). On the other  hand,    $X \varsubsetneq  \widetilde X \varsubsetneq H(U_d(X))$, since  $\widetilde X=Y\times \ell _2$ is not Bourbaki-complete (Proposition  \ref{Samuel versus completion}). Finally $\widetilde X\neq \upsilon X$ and $s_d X\neq \beta X$ since $\widetilde X$ is not realcompact and $X$ is not  $UC$. \end{example}

We conclude the paper by linking  the Hewitt-Nachbin realcompactification of a metrizable space with the family of all its  Samuel realcompactifications given by  compatible metrics. This result is in the line of that one given by Woods in the frame of compactifications. Namely, he proved in \cite{woods} that   if $(X,\tau)$ is a topological  metrizable space, then its Stone-\v{C}ech compactification is the supremum of the family of the Samuel compactifications defined by compatible metrics, that is, $$\beta X =\bigvee \big\{s_d X:  d \,\, metric \,\, with \,\,  \tau_d=\tau \}.$$

Then in this line we have the following.

\begin{theorem} \label{upsilon as supremum Samuel} Let $(X,\tau)$ be a topological metrizable space. Then,
$$\upsilon X =\bigvee \big\{H(U_d(X)): d\text{ metric with \,}  \tau_d=\tau\}.$$
\end{theorem}
\begin{proof} First note that $\upsilon X$ is greater than the above supremum, since for any compatible metric $d$ we have $\upsilon X \geq H(U_d(X))$. Finally, in order to see that the reverse inequality holds,   it is enough to see that every real continuous function on $X$ can be continuously extended to this supremum. Indeed, take $f\in C(X)$ and $d_0$ a metric on $X$ defining the topology $\tau$ (recall that $X$ is metrizable), then the metric $d(x,y)=d_0(x,y)+|f(x)-f(y)|$ is compatible, and $f\in U_d(X)$ (in fact, $f\in Lip_d(X)$). Finally since $f$ can be continuously extended to  $H(U_d(X))$, then it can be also extended to the supremum, as we wanted.
\end{proof}

Note that above proof works also to derive that $$\upsilon X =\bigvee \big\{H(Lip_d(X)): d \,\, metric \,\, with \,\,  \tau_d=\tau\}.$$
And this is another way to make clear the difference  between $\upsilon X$ and $H(U_d(X))$ for a given metric space $(X,d)$. Namely, if $\stackrel{{t}}{\sim}$ denotes "topologically equivalent", then  $$\upsilon X =\bigvee \big\{H(Lip_\rho(X)): \rho\stackrel{{t}}{\sim} d\big\} \text{ \,\,\,\,\,\and\,\,\,\,\,} H(U_d(X)) =\bigvee \big\{H(Lip_\rho(X)): \rho\stackrel{{u}}{\sim} d\big\}.$$
\medskip

\centerline{\sc Acknowledgements}

\vspace{.1cm}

We would like to thank  the referee for his/her interesting suggestions, and for providing us valuable references.

\end{document}